\documentclass[12pt]{amsart}
\usepackage{amssymb,epsf}
\usepackage{amsfonts,amssymb}
\usepackage{amsmath,bbm}

\usepackage{amssymb}
\usepackage{amsmath}
\usepackage[english]{babel}
\usepackage{amsfonts}

\newtheorem{thm}{Theorem}[section]
\newtheorem{lemma}[thm]{Lemma}
\newtheorem{cor}[thm]{Corollary}
\newtheorem{prop}[thm]{Proposition}

\theoremstyle{definition}

\newtheorem{remk}[thm]{Remark}
\newtheorem{exam}[thm]{Example}
\newtheorem{conj}{Conjecture}

\newtheorem*{thm*}{Theorem}
\newtheorem*{prop*}{Proposition}

\newcommand{\hide}[1]{}

\newcommand{\cyc}[1]{\mathbb{Z}/{#1}}
\newcommand{\Un}[1]{\mathbf{1}_{#1}}
\newcommand{\set}[1]{\left\{#1\right\}}
\newcommand{\diag}[1]{\mathrm{diag}\left(#1\right)}
\newcommand{\norm}[1]{\left\Vert#1\right\Vert}

\newcommand{\cA}{\mathcal{A}}
\newcommand{\cB}{{\mathcal{B}}}
\newcommand{\cBp}{{\widehat{\cB}^+}}

\newcommand{\cK}{{\mathcal{K}}}
\newcommand{\cD}{{\mathcal{D}}}

\newcommand{\cC}{{\mathcal{C}}}

\newcommand{\cN}{{\mathcal{N}}}
\newcommand{\cM}{{\mathcal{M}}}
\newcommand{\cP}{{\mathcal{P}}}

\newcommand{\cS}{{\mathcal{S}}}

\newcommand{\cQ}{{\mathcal{Q}}}
\newcommand{\cU}{{\mathcal{U}}}
\newcommand{\cV}{{\mathcal{V}}}

\newcommand{\bR}{{\mathbf{R}}}
\newcommand{\bC}{{\mathbf{C}}}
\newcommand{\bK}{{\mathbf{K}}}

\newcommand{\dC}{{\mathbb{C}}}

\newcommand{\dN}{{\mathbb{N}}}
\newcommand{\dQ}{{\mathbb{Q}}}
\newcommand{\dR}{{\mathbb{R}}}
\newcommand{\dT}{{\mathbb{T}}}
\newcommand{\dZ}{{\mathbb{Z}}}

\newcommand{\gog}{\mathfrak{g}}
\newcommand{\goh}{\mathfrak{h}}
\newcommand{\gp}{\mathfrak{p}}
\newcommand{\gP}{\mathfrak{P}}

\newcommand{\gA}{\mathfrak{A}}

\newcommand{\gX}{\mathfrak{X}}
\newcommand{\ov}{\overline}

\newcommand{\Her}{\operatorname{Hom}}
\newcommand{\Id}{\operatorname{Id}}
\newcommand{\Gal}{\operatorname{Gal}}
\newcommand{\Lin}{{\operatorname{Lin}}}
\newcommand{\Orb}{{\operatorname{Orb}}}

\newcommand{\Res}{{\operatorname{Res}}}

\newcommand{\supp}{\operatorname{supp}}

\newcommand{\tr}{\mathfrak{tr}}
\newcommand{\id}{\operatorname{id}}

\newcommand{\img}{\mathbf{i}}
\newcommand{\eps}{\varepsilon}

\def\RR{\mathbb{R}}

\def\QQ{\mathbb{Q}}
\def\ZZ{\mathbb{Z}}
\def\NN{\mathbb{N}}

\def\eps{\varepsilon}

\def\id{\mathrm{id}}
\def\A{\mathcal{A}}
\def\cB{\mathcal{B}}

\DeclareMathOperator\res{Res}
\DeclareMathOperator\ind{Ind}
\DeclareMathOperator\Ind{Ind}
\DeclareMathOperator\lc{lc}

\newtheorem{comment}{Comment}

\begin{document}

\author{Jaka Cimpri\v c}
\address{University of Ljubljana, Faculty of Mathematics and Physics, Department of Mathematics, Jadranska 21, SI-1000 Ljubljana, Slovenija}
\email{Jaka.Cimpric@fmf.uni-lj.si}

\author{Yurii Savchuk}
\address{FAU Erlangen-N\"urnberg, Department Mathematik, Cauerstr. 11, 91058 Erlangen, Germany}
\email{savchuk@math.fau.de, savchuk@math.uni-leipzig.de}

\subjclass[2000]{14P99, 14A22, 47Lxx, 13J25, 06F25, 12J15}

\keywords{Quadratic module, involution, induced representation, Weyl algebra, central simple algebra, real field, Galois extension}

\title{Induced quadratic modules}

\date{\today}

\begin{abstract}
Positivity in $\ast$-algebras can be defined either algebraically, by quadratic modules, or analytically, by $\ast$-re\-pre\-sen\-ta\-ti\-ons. 
By the induction procedure for $\ast$-representations we can lift the analytical notion of positivity from a $\ast$-subalgebra to the entire $\ast$-algebra.
The aim of this paper is to define and study the induction procedure for quadratic modules. The main question is when a given quadratic module on the $\ast$-algebra 
is induced from its intersection with the $\ast$-subalgebra. This question is very hard even for the smallest quadratic module 
(i.e. the set of all sums of hermitian squares) and will be answered only in very special cases.
\end{abstract}

\maketitle

\section{Introduction}

Non-commutative real algebraic geometry studies real and complex associative algebras with involution from a geometric perspective.
Hermitian elements are considered as non-commutative real polynomials and ``well-behaved'' (not necessarily bounded) 
$\ast$-representatons as non-commutative real points, see \cite{sch2}. For every finite set $S$ of hermitian elements one tries 
to understand the set  of all hermitian elements that are positive (semi-)definite in every well-behaved  $\ast$-representation
in which all elements from $S$ are positive semi-definite. An algebraic description of this set is called a Positivstellensatz for $S$. 
They are well-understood in the commutative case, see \cite{mar} for a survey, but in the non-commutative case they are very difficult to find. 
Most algebras with involution for which Positivstellens\" atze are known (e.g. Weyl algebras, quantum polynomials, matrix polynomials,
central simple algebras) have the property that their well-behaved $\ast$-representations are induced 
from one-dimensional $\ast$-representations on an appropriate commutative subalgebra.
However, the quest for general ``Induced Positivstellens\" atze'' is still on.

This paper addresses only a small part of the big problem. We define and study the induction procedure for various order-theoretic structures 
(i.e. quadratic modules, preorderings, orderings) related to Positivstellens\" atze. 

In Section \ref{secprelim} we discuss the relationship between $\ast$-representations and quadratic modules whose special case is 
the relationship between bounded $\ast$-representations and archimedean preorderings. We recall the definition of a bimodule projection 
from a $\ast$-algebra to its $\ast$-sub\-al\-geb\-ra and the corresponding induction procedure for $\ast$-representa\-ti\-ons.

In Section \ref{secindmod} we extend the induction procedure from $\ast$-representa\-ti\-ons to quadratic modules. 
Let $\cB$ be a $\ast$-subalgebra of a $\ast$-algebra $\cA$ and let $p \colon \cA \to \cB$ be a bimodule projection 
(i.e. a $\cB$-$\cB$ bimodule map which commutes with the involution and preserves identity). For every quadratic module $\cN$ in $\cB$ we define the set
$$\ind\cN:=\set{a\in\cA \mid a=a^\ast \mbox{ and } p(x^*ax)\in\cN\ \mbox{for all}\ x\in\cA}$$
which is a quadratic module in $\cB$ if and only if it contains $\Un{}$. If $\cN$ is a preordering, then $\ind \cN$ 
need not be a preordering even if it contains $\Un{}$. For every quadratic module $\cM$ in $\cA$ the set 
$$\res \cM := \{p(m) \mid m \in M\}$$

is a quadratic module in $\cB$ and the following questions make sense:
\begin{enumerate}
\item[(Q1)] Is $\cM=\ind \res \cM$?
\item[(Q2)] Is $\res \cM=\cM \cap \cB$? 
\item[(Q3)] Is $\cM$ generated by $\cM \cap \cB$?
\end{enumerate}
These questions are surprisingly hard even for very special $p$ and $\cM$.

The bimodule projections that we are interested in come either from group actions 
via $\ast$-automorphisms (we discuss matrix $\ast$-algebras, regular functions on affine varietes, Galois extensions of fields)
or group gradings (we discuss group algebras, cyclic algebras, quantum polynomials, Weyl algebra). 
Interesting quadratic modules include $\sum \cA^2$, the set of all sums of hermitian squares, and $\cA_+$, 
the set of all hermitian elements that are positive semi-definite in all well-behaved $\ast$-representations.

For example, let $p$ be the ``natural'' bimodule projection from the Weyl algebra
$\cA=\dC\langle a, a^\ast \mid a a^\ast-a^\ast a=1\rangle$
to its subalgebra $\cB=\dC[a^*a]$. We consider the Fock-Bargmann $\ast$-representation
as the only well-behaved $\ast$-representation.
For $\cA_+$, the answers to (Q1) and (Q2) turn out to be positive, while the answer to (Q3)
turns out to be negative. For $\sum \cA^2$, (Q1) is still open while the answers to (Q2)
and (Q3) are trivially positive. For other $\ast$-algebras we have similar situations.

Finally, in the Appendix we extend our construction of $\ind$ and $\res$ of quadratic modules
from bimodule projections to rigged $\cA$-$\cB$ bimodules.

\section{Preliminaries}\label{secprelim}

In this section we explain the relationship between quadratic modules and $\ast$-representations.
We also recall the construction of an induced $\ast$-representation via a bimodule projection.

\subsection{Quadratic modules}

Let $\cA$ be a $\ast$-algebra over a field $F$. This means that $\cA$ is an associative unital algebra over $F$ 
and $\cA$ is equipped with an involution $\ast$ such that $F^\ast \subseteq F$ ($F$ is identified with $F \cdot \Un{\cA}$). 
We will denote by $\cA_h$ the set $\{a \in \cA \mid a^\ast=a\}$ of all hermitian elements of $\cA$. 
In the following we will also assume that the field $F_h:=F \cap \cA_h$ is equipped with an ordering $\ge$ 
such that $f f^\ast \ge 0$ for every $f \in F$. (For $F=\dC$ this assumption implies that $\ast|_\dC$ is the conjugation. )

A subset $\cM\subseteq\cA_h$ is called a \emph{quadratic module (q.m.)} in $\cA$ if the following axioms are satisfied:
\begin{enumerate}
\item[(QM1)]\quad $\cM+\cM \subseteq \cM$ and $\lambda\cM\subseteq\cM$ for all $\lambda \in F^{\ge 0},$
\item[(QM2)]\quad $c^*\cM c \subseteq \cM$ for every $c \in \cA$,
\item[(QM3)]\quad $\Un{\cA}\in\cM.$
\end{enumerate}
We say that $\cM$ is \emph{proper} if $-\Un{\cA}\not\in \cM.$ For a subset $S\subseteq\cA_h$ 
we denote by $QM_\cA(S)$ the q.m. generated by $S,$ that is 
$$QM_\cA(S)=\{\sum_{i=1}^m \lambda_i a_i^*s_i^{}a_i^{}\mid s_i\in S,\ a_i\in\cA,\ \lambda_i \in F^{\ge 0}\}.$$
The set $\sum \cA^2=QM_\cA(\set{\Un{\cA}})$ is the smallest quadratic module in $\cA$. In all our examples $\sum \cA^2$ is proper. 
The largest q.m. in $\cA$ is the set of all hermitian elements $\cA_h.$ It is never proper. 

\begin{remk}
Note that for $F=\QQ$, $F=\dR$ and $F=\dC$ the second part of (QM1) follows from (QM2) and (QM3), so it can be omitted.
In particular, we can also omit $\lambda_i$ in the definition of $QM_\cA(S)$.
\end{remk}

A quadratic module $\cM$ in $\cA$ is \emph{archimedean} if for every $a \in \cA_h$ there is $n \in \NN$
such that $n \Un{\cA}-a \in \cM$. It is a \emph{preordering} if for every $x,y \in \cM$ such that
$xy=yx$ we have that $xy \in \cM$. Note that $F^{\ge 0}$ is a preordering. 
For $F=\QQ$, $F=\dR$ and $F=\dC$, $F^{\ge 0}$ is also archimedean.

\subsection{$\ast$-Representations}\label{astrep}

Let $\cV$ be an inner product space over $F$ and let $L(\cV)$ be the set of all linear operators on $\cV.$
A non-zero algebra homomorphism $\pi:\cA\to L(\cV)$ is called a \emph{$*$-representation} of $\cA$ on $\cV$ 
if $\langle\pi(b)v,w\rangle=\langle v,\pi(b^*)w\rangle$ and $\pi(\Un{\cA})v=v$ for all $v,w\in\cV.$ 
We will also write $\cD(\pi)$ for $\cV.$ 
A $*$-representation is called \emph{cyclic} if it has a \textit{cyclic vector},
i.e. if there is $v\in\cD(\pi)$ such that $\set{\pi(a)v\mid a\in\cA}=\cD(\pi)$.

Let $\pi$ be a $*$-re\-pre\-sen\-tation of $\cA$ on $\cV.$ One easily checks that
$$
QM_\cA(\pi):=\set{a\in\cA_h\ \mid\ \langle \pi(a)v,v\rangle\geq 0 \ \mbox{for all}\ v\in\cV}
$$
is a q.m. in $\cA.$ Note that $QM_\cA(\pi)$ is archimedean if and only if $\pi$ is a bounded $\ast$-representation 
(i.e. $\pi(a)$ is bounded for every $a \in \cA$). If $QM_\cA(\pi)$ is archimedean, it is also a preordering.

A non-zero linear functional $\varphi$ on $\cA$ is called \emph{positive} if $\varphi(\sum\cA^2)\geq 0$. 
It is called \emph{hermitian} if $\varphi(a^\ast)=\varphi(a)^\ast$ for every $a \in \cA$.
By the GNS-construction (see \cite[Section 8.6]{s_book}) for each positive hermitian linear functional $\varphi$ 
there exists a cyclic $*$-representation $\pi_\varphi$ with a cyclic vector $\xi_\varphi\in\cD(\pi_\varphi)$ 
such that $\varphi(a)= \langle \pi_\varphi(a)\xi_\varphi,\xi_\varphi \rangle$ for every $a \in A$. 

Every $*$-algebra $\cA$ over $\dC$ will be considered with the finest locally convex topology (defined by the family of all seminorms on $\cA$). 
Closed quadratic modules in $\cA$ play an important role in the real algebraic geometry, see \cite{mar,sch2}.

\begin{prop}\label{prop_closed_qm}
A quadratic module $\cM$ in a complex $\ast$-algebra $\cA$ is closed if and only if $\cM=QM_\cA(\pi)$ for some $*$-representation $\pi$ of $\cA.$
\end{prop}
\begin{proof}
Let $\pi$ be a $*$-representation of $\cA$ on an inner-product space $\cV.$ For each $v\in\cV$ denote by $\cA_v$ the set 
$\set{a\in\cA_h\ \mid\ \langle\pi(a)v,v\rangle\geq 0}.$ The set $\cA_v$ is closed, thus $QM_\cA(\pi)=\cap_{v\in\cV}\cA_v$ is closed.

Let $\cM\subseteq\cA_h$ be a closed q.m. and let $a\in\cA_h\backslash\cM.$ By the separation theorem for closed convex sets
there exists a $\dR$-linear functional $\varphi_a$ on $\cA_h$ such that $\varphi_a(\cM)\geq 0$ and $\varphi_a(a)<0.$ 
We denote the corresponding $\dC$-linear functional on $\cA$ by $\varphi_a.$ Let $\pi_a$ be the GNS representation of 
$\varphi_a$ and let $\pi:=\oplus_{a}\pi_a.$ Then $\cM=QM_\cA(\pi)$ by construction.
\end{proof}

\subsection{Inclusion between quadratic modules}

This subsection settles an important basic question of non-commutative real algebraic geometry. It will only be used in Example \ref{discex}. 

Let $\rho$ and $\pi$ be $*$-representations of a complex $*$-algebra $\cA.$ We say that $\rho$ \emph{is weakly contained in} 
$\pi$ if every functional $\omega_{\rho,\psi}(a):=\langle \rho(a)\psi,\psi \rangle$
can be weakly approximated by functionals $\sum_{i=1}^m\omega_{\pi,\xi_i}.$ 
That is, for arbitrary $\varepsilon>0$ and $a_1,\dots,a_n\in\cA$ there exist $\xi_1,\xi_2,\dots,\xi_m\in\cD(\pi)$ such that
\begin{gather}\label{eq_weak_appr}
|\omega_{\rho,\psi}(a_j)-\sum_{i=1}^m\omega_{\pi,\xi_i}(a_j)|<\varepsilon,\ \mbox{for all}\ j=1,\dots,n.
\end{gather}

\begin{prop}\label{prop_weakly}
A $*$-representation $\rho$ is weakly contained in $\pi$ if and only if $QM_\cA(\pi)\subseteq QM_\cA(\rho).$
\end{prop}
\begin{proof}
Let $\rho$ be weakly contained in $\pi$ and let $a\in QM_\cA(\pi).$ Then $\omega_{\pi,\xi}(a)\geq 0$ for all $\xi\in\cD(\pi)$ and by (\ref{eq_weak_appr}) we have $\omega_{\rho,\psi}(a)\geq 0$ for all $\psi\in\cD(\rho).$ Hence $a\in QM_\cA(\rho),$ which proves the ``only if'' direction.

To prove the ``if'' direction, we use some basics from the theory of locally convex spaces. Let $\cP(\pi)$ be the convex hull of $\set{\omega_{\pi,\xi}\mid\norm{\xi}=1}.$ Then for $\varphi\in\cP(\pi)$ we have $\varphi(\Un{})=1$ and 
$$
QM_\cA(\pi)=\set{a\in\cA_h\mid\varphi(a)\geq 0,\ \forall\varphi\in\cP(\pi)}.
$$
Consider $\cA_h$ as an $\dR$-vector space and let $\cA_h'$ be the space of all linear functionals on $\cA_h,$ so that $\langle\cA_h,\cA_h'\rangle$ is a dual pair. Consider $\cP(\pi)$ as a subset of $\cA_h'$ and let
$$\cP(\pi)^\circ=\set{a\in\cA_h\mid\varphi(a)\leq 1,\ \forall\varphi\in\cP(\pi)}$$
be its polar set in $\cA_h.$ Fix an arbitrary $\omega_{\rho,\psi},\ \psi\in\cD(\rho)$ and assume without loss of generality that $\norm{\psi}=1.$ Then for $a\in\cA_h$ we have
\begin{gather*}
a\in\cP(\pi)^\circ\Rightarrow\forall\varphi\in\cP(\pi)\ \varphi(a)\leq 1\Rightarrow\forall\varphi\in\cP(\pi)\ \varphi(\Un{}-a)\geq 0\Rightarrow\\
\Rightarrow  \Un{}-a\in QM_\cA(\pi)\Rightarrow \Un{}-a\in QM_\cA(\rho) \Rightarrow \omega_{\rho,\psi}(a)\leq 1,
\end{gather*}
that is $\omega_{\rho,\psi}$ is contained in $(\cP(\pi)^{\circ})^\circ.$ By the bipolar theorem $\omega_{\rho,\psi}$ is weakly approximated by the elements of $\cP(\pi).$
\end{proof}

\subsection{Bimodule projections}

Let $\cB\subseteq\cA$ be $*$-algebras over a field $F$. A mapping $p \colon \cA\to \cB$ 
is said to be a \emph{bimodule projection}
if it satisfies the following properties:
\begin{enumerate}
\item[(CE1)] $p(a_1+a_2)=p(a_1)+p(a_2)$ for every $a_1,a_2 \in \cA$,
\item[(CE2)] $p(b_1 a b_2)=b_1 p(a) b_2$ for every $b_1,b_2 \in \cB$ and $a \in \cA$,
\item[(CE3)] $p(a^\ast)=p(a)^\ast$,
\item[(CE4)] $p(\Un{\cA})=\Un{\cB}$.
\end{enumerate}
A bimodule projection is a \emph{conditional expectation} if it satisfies
\begin{enumerate}
\item[(CE5)] $p(\sum \cA^2) \subseteq \sum \cA^2$ (or equivalently $p(\sum \cA^2) = \sum \cA^2 \cap \cB$.)
\end{enumerate}
We will not need properties (CE4) and (CE5) until subsection \ref{subsecres}.

\begin{exam}
Every hermitian linear functional $\phi \colon \cA \to F$ satisfies (CE1)-(CE3) and $\frac{1}{\phi(\Un{})} \, \phi$ is a bimodule projection. If also
$\phi$ is positive and $F^{\ge 0}=\sum F^2$, then $\frac{1}{\phi(\Un{})} \, \phi$
is a conditional expectation.
\end{exam}

\noindent Many more conditional expectations will be given in Sections \ref{subsect_aver} and \ref{subsect_grad}.

\begin{exam}
Let $\gog$ be the Heisenberg Lie algebra, i.e. the real Lie algebra with generators $a,b,c$, relations $[a,b]=c,\ [a,c]=[b,c]=0$ 
and involution $a^*=-a,\ b^*=-b,\ c^*=-c$, and let $\goh$ be its Lie subalgebra generated by $a,c$. Further, let $\cA$ and $\cB$ 
denote the universal enveloping algebras $\cU(\gog)$ and $\cU(\goh)$ respectively. We claim that there exists no bimodule projection from $\cA$ to $\cB.$ 
Assume to the contrary, that $p:\cA\to\cB$ is a bimodule projection. Then $c=p(c)=p(ab-ba)=ap(b)-p(b)a=0,$ by commutativity of $\cB.$ A contradiction.
\end{exam}

\subsection{Induced $\ast$-representations}

Let $\cB\subseteq\cA$ be $*$-algebras over a field $F$ and let $p:\cA\to\cB$ be a bimodule projection. 
Further, let $\rho$ be a $*$-re\-pre\-sen\-tation of $\cB$ on an inner-product space $\cV.$ Let us briefly recall 
the construction of the induced representation $\ind\rho$, see \cite{ss}. Let $\cA\otimes_\cB\cV$ denote 
the quotient space of $\cA\otimes\cV$ by the linear hull of $\set{ab\otimes v-a\otimes\rho(b)v,\ a\in\cA,\ b\in\cB,\ v\in\cV}.$ 
Note that $\cA\otimes_\cB\cV$ is an induced $\cA$-module in the sense of \cite{higman}. A representation $\rho$ is called \emph{inducible} 
if the sesquilinear form $\langle\cdot,\cdot\rangle_0$ on $\cA\otimes_\cB\cV$ defined by
$$
\langle\sum_i x_i\otimes v_i,\sum_j y_j\otimes w_j\rangle_0:=\sum_{i,j}\langle\rho(p(y_j^*x_i^{})) v_i,w_j\rangle,\ x_i,y_j\in\cA,\ v_i,w_j\in\cV
$$
is positive semi-definite. The kernel $\cN_\rho$ of the form $\langle\cdot,\cdot\rangle_0$ is invariant under the action of $\cA$ on $\cA\otimes_\cB\cV.$ 
For $a\otimes v\in\cA\otimes_\cB\cV$ denote by $[a\otimes v]\in\cA\otimes_\cB\cV/\cN_\rho$ the image under the quotient mapping. 
Thus, $\cA\otimes_\cB\cV/\cN_\rho$ is an inner-product space and there is a well-defined $*$-representation $\pi=\Ind\rho$ of $\cA$ on the unitary space $\cA\otimes_\cB\cV/\cN_\rho$ which acts by
$$
\pi(a)(\sum_i[x_i\otimes v_i]):=\sum_i[ax_i\otimes v_i].
$$

\begin{exam}\label{gnsind}
Let $\phi$ be a hermitian linear functional from $\cA$ to $F$.
The left regular representation $\lambda$ of $F$ is inducible if and only if $\phi$ is positive.
In this case $\ind \lambda=\pi_\phi$, i.e. the $\ast$-representation associated to $\phi$ by the 
GNS construction. Note also that $QM_F(\lambda)=F^{\ge 0}$ and
$
QM_\cA(\pi_\varphi)=\set{a\in\cA_h\ \mid\ \varphi(x^\ast a x) \geq 0 \ \mbox{for all}\ x \in\cA}.
$
\end{exam}

\section{Induced quadratic modules}\label{secindmod}

\subsection{Definition}
Let $\cB\subseteq\cA$ be a $*$-algebras over a field $F$. Motivated by Example \ref{gnsind}
we define for every quadratic module $\cN\subseteq\cB$ 
\begin{gather*}
\Ind\cN:=\set{a\in\cA_h \mid p(x^*ax)\in\cN\ \mbox{for all}\ x\in\cA}.
\end{gather*}
The set $\Ind\cN$ is called the \emph{induction of $\cN$ from $\cB$ to $\cA$ via $p$}.
We will also write $\Ind_{\cB\uparrow\cA}\cN$ or $\Ind^p\cN$ for $\Ind\cN$.

\begin{prop}\label{indmoddef}
Let $\cN\subseteq\cB$ be a q.m. Then
\begin{enumerate}
  \item[(i)] the set $\Ind\cN$ is a quadratic module if and only if
\begin{gather}\label{eq_indcb1}
\set{p(x^*x), x\in\cA}\subseteq\cN,
\end{gather}
  \item[(ii)] if $(\ref{eq_indcb1})$ is satisfied, then $\Ind\cN$ is proper if and only if $\cN$ is proper,
  \item[(iii)] if $(\ref{eq_indcb1})$ is satisfied, then $\Ind\cN$ is the largest q.m. contained in $p^{-1}(\cN).$
\end{enumerate}
\end{prop}
\begin{proof}
 (i): One can easily check that $\Ind\cN$ is a subset of $\cA_h$ and (QM1),(QM2) are satisfied. Then $p(x^*x)\in\cN$ for all $x\in\cA$ if and only if $\Un{\cA}\in\Ind\cN,$ i.e. (QM3) is satisfied.

(ii): Clearly, if $\cN=\cB_h$ then $\Ind\cN=\cA_h$ is also not proper. Assume now $\Ind\cN=\cA_h,$ that is $-\Un{\cA}\in\Ind\cN.$ Then $-\Un{\cB}=p(\Un{\cA}^*(-\Un{\cA})\Un{\cA})\in\cN,$ i.e. $\cN$ is not proper.

(iii): Follows from the definition of $\Ind\cN.$
\end{proof}

In the view of Proposition \ref{indmoddef} we say that a q.m. $\cN\subseteq\cB$ is \emph{inducible} 
if and only if (\ref{eq_indcb1}) is satisfied. Equivalently, $\cN$ is inducible if and only if 
$\Ind\cN$ is a q.m.

In general, induction does not respect preorderings:

\begin{exam}
Let $\cA=\dR[x,y]$ and $\cM=QM_\cA(\{x,y\})$. By \cite{mar}, Example 4.1.5 and Theorem 4.1.2, $\cM$ is closed. (The proof of this fact does not use Proposition \ref{prop_closed_qm}.)
A direct computation shows that $xy \not\in \cM$. Hence there exists a linear functional $\varphi:\cA\to\dR$ such that $\varphi(\cM)\geq 0$ and $\varphi(xy)<0.$ For 
$p=\frac{1}{\phi(\Un{})}\phi$ we have that $x,y\in\Ind^p\dR^{\ge 0}$ and $xy\notin\Ind^p\dR^{\ge 0}$. Therefore $\Ind^p\dR^{\ge 0}$ is not a preordering.
\end{exam}

If $\sum \cA^2$ is archimedean then induction respects closed preorderings. 
In view of subsection \ref{astrep}, this will follow from Proposition \ref{prop_ind_rep_mod}.

\subsection{Relation to induced $\ast$-representations}

Let $\cN$ be a closed q.m. in a real or complex $\ast$-algebra $\cB$. By Proposition \ref{prop_closed_qm} 
there exists a $*$-representation $\rho$ of $\cB$ such that $\cN=QM_\cB(\rho).$ 
By the proof of Proposition \ref{prop_closed_qm}, we may assume 
that $\rho$ is a direct sum of cyclic $*$-representations.
Proposition \ref{prop_ind_rep_mod} will imply that $\ind \cN$ is also closed.

\begin{prop}\label{prop_ind_rep_mod}
If $\rho$ is a direct sum of cyclic $\ast$-representations then
$QM_\cB(\rho)$ is inducible if and only if $\rho$ is inducible. Moreover if $QM_\cB(\rho)$ is inducible, then $QM_\cA(\Ind\rho)=\Ind(QM_\cB(\rho)).$
\end{prop}
\begin{proof}
Let $\cV=\oplus_i\cV_i$ be a decomposition of $\cV$ into orthogonal sum of cyclic components and $v_i\in\cV_i$ be the corresponding cyclic vectors. Assume that $QM_\cB(\rho)$ is inducible, that is $\set{p(a^*a)\ \mid\ a\in\cA}\subseteq QM_\cB(\rho)$ or, equivalently, $\langle\rho(p(a^*a))w,w\rangle\geq 0$ for all $a\in\cA,\ w\in\cV.$ We show that $\rho$ is inducible. For let $w_i\in\cV$ and $a_i\in\cA.$ Then $w_i=\sum_{ij}\rho(b_{ij})v_j$ and we get
\begin{gather}
\nonumber\langle \sum_i a_i\otimes w_i,\sum_i a_i\otimes w_i\rangle_0=\sum_{i,j}\langle\rho(p(a_j^*a_i^{}))w_i,w_j\rangle=\\
\label{eq_long}=\sum_{i,j}\sum_{k,m}\langle\rho(p(a_j^*a_i^{})) \rho(b_{im})v_m,\rho(b_{jk})v_k\rangle=\\
\nonumber=\sum_{k}\sum_{i,j}\langle\rho(p(b_{jk}^*x_j^*x_i^{}b_{ik}^{})) v_k,v_k\rangle=\sum_{k}\langle\rho(p(z_k^*z_k^{})) v_k,v_k\rangle\geq 0,
\end{gather}
where $z_k=\sum_i x_ib_{ik}.$ Thus $\rho$ is inducible. The converse statement is trivial.

Take $a\in\Ind(QM_\cB(\rho)).$ Then for all $x\in\cA$ and $w\in\cV$ holds $\langle \rho(p(x^*ax))w,w\rangle\geq 0.$ A calculation similar to (\ref{eq_long}) shows that
$$\langle\sum_i ax_i\otimes w_i,\sum_i x_i\otimes w_i\rangle_0\geq 0$$
for arbitrary $w_i\in\cV$ and $x_i\in\cA.$ Thus $\Ind\rho(a)$ is positive or, equivalently $a\in QM_\cA(\Ind\rho).$ The reverse inclusion is trivial.
\end{proof}

\subsection{Restriction}\label{subsecres}

Let $\cB\subseteq\cA$ be $*$-algebras and $\pi$ be a $\ast$-representation of $\cA$. The restriction of $\pi$ to $\cB$ is defined by $\res \pi=\pi|_\cB$. 
It is easy to see that $QM_\cB(\res \pi)=QM_\cA(\pi) \cap \cB$. The most natural extension to quadratic modules would be that $\cM \cap \cB$ is considered as the restriction of a quadratic module $\cM$ on $\cA$. However, we will not go this way.

Let $p:\cA\to\cB$ be a fixed bimodule projection. For a quadratic module $\cM$ on $\cA$ we define
\begin{gather*}
\Res\cM:=\set{p(m) \mid m\in\cM}.
\end{gather*}
The set $\Res\cM$ is called the \emph{restriction of $\cM$ onto $\cB$ via $p$}. It can be easily checked that $\Res\cM\subseteq\cB$ is always a q.m. (property (CE4) is used here for the first time), it is always inducible 
and it always contains $\cM \cap \cB$.

The advantage of this definition of $\res \cM$ is that we obtain a Galois correspondence between quadratic modules on $\cA$ and inducible quadratic modules on $\cB$. The following is easy to verify:
\begin{prop}\label{prop_indres}
Let $p \colon \cA \to \cB$ be a bimodule projection. For every q.m. $\cM\subseteq\cA$
and every inducible q.m. $\cN\subseteq\cB$ we have that
$$\res \cM \subseteq \cN \mbox{  iff  } \cM \subseteq \ind \cN.$$
Here are some special cases:
\begin{enumerate}
\item[(1)] $\Ind(\Res\cM) \supseteq \cM$ and $\Res(\Ind\cN) \subseteq \cN$.
\item[(2)] $\res \cM \subseteq \cM$ iff $\cM \subseteq \ind (\cM \cap \cB).$
\item[(3)] $\cN \subseteq \ind \cN$ iff $\res QM_\cA(\cN) \subseteq \cN$.
\end{enumerate}
\end{prop}

By (1), $\Res(\Ind(\Res\cM)) = \Res\cM$ and $\Ind(\Res(\Ind\cN)) = \Ind\cN.$ We say that a q.m. $\cM\subseteq\cA$ is \emph{induced} if $\cM=\ind \cN$ for some q.m. $\cN \subseteq \cB$ (or equivalently, if $\cM=\ind \res \cM$). We say that a q.m. $\cN \subseteq \cB$ is \emph{restricted} if $\cN=\res \cM$ for some q.m. $\cM\subseteq\cA$ (or equivalently, if $\cN=\Res(\Ind\cN)$).

By (2) we have that $\res \cM = \cM \cap \cB$ iff $\cM \subseteq \ind (\cM \cap \cB).$ It follows that $\cM=\ind(\cM \cap \cB)$ iff $\res \cM = \cM \cap \cB$ and $\cM=\ind \res \cM$.

By (3) we have that every inducible q.m. $\cN \subseteq \cB$ which satisfies $\cN \subseteq \ind \cN$ is restricted (since $QM_\cA(\cN) \cap \cB \subseteq \res QM_\cA(\cN) \subseteq \cN \subseteq QM_\cA(\cN) \cap \cB$). 
In many cases (see Propositions \ref{prop_aver} and \ref{prop_graded}), we also have the converse: every restricted q.m. $\cN \subseteq \cB$ satisfies $\cN \subseteq \ind \cN$.

We say that an inducible q.m. $\cN\subseteq\cB$ is \emph{perfect} if
$$QM_\cA(\cN)=\Ind\cN.$$
In the situations from Propositions \ref{prop_aver} and \ref{prop_graded}, every perfect q.m. $\cN$ is restricted and $\ind \cN$ satisfies (Q1)-(Q3) by the discussion above.


\subsection{Induction in stages}

Induction in stages is a useful tool in the theory of induced representations. The next proposition is a counterpart of it for quadratic modules. Note that a compositum of two bimodule projections is always a bimodule projection.
\begin{prop}\label{prop_ind_stage}
Let $\cC\subset\cB\subset\cA$ be $*$-algebras, let $p_1:\cA\to\cB,\ p_2:\cB\to\cC$ be bimodule projections and let $\cK\subseteq\cC$ be a q.m. 
$$
\Ind_{\cB\uparrow\cA}(\Ind_{\cC\uparrow\cB}\cK)=\Ind_{\cC\uparrow\cA}(\cK).
$$
\end{prop}
\begin{proof}
Take $x\in\cA_h.$ Using properties (CE1)-(CE4) we get
\begin{gather*}
x\in\Ind_{\cB\uparrow\cA}(\Ind_{\cC\uparrow\cB}\cK)\Longleftrightarrow p_1(a^*xa)\in\Ind_{\cC\uparrow\cB}\cK,\ \forall a\in\cA\Longleftrightarrow\\
\Longleftrightarrow p_2(b^*p_1(a^*xa)b)=p_2(p_1(b^*a^*xab))\in\cK,\ \forall a\in\cA,\forall b\in\cB\Longleftrightarrow\\
\Longleftrightarrow p_2(p_1(a^*xa))=p_2(p_1(\Un{\cB}a^*xa\Un{\cB}))\in\cK,\ \forall a\in\cA \Longleftrightarrow x\in\Ind_{\cC\uparrow\cA}\cK.
\end{gather*}
This completes the proof. 
\end{proof}

\begin{exam}\label{dce}
In this example we show that a composition of two conditional expectations is in general not a conditional expectation. For let $\cA=\dR[x],\ \cB=\dR[x^2],\ \cC=\dR[x^4]$ and let $\gp_1:\cA\to\cB,\ \gp_2:\cB\to\cC$ be bimodule projections defined by $(\gp_1(f))(x)=\frac 1 2(f(x)+f(-x))$ and $(\gp_2(g))(x^2)=\frac 1 2(g(x^2)+g(-x^2))$ respectively. It easy to verify that $\gp_1,\gp_2$ are conditional expectations. On the other hand $(\gp_2\circ\gp_1)((x^3-x)^2)=-2x^4\notin\sum\cA^2,$ i.e. $(\gp_2\circ\gp_1)$ is not a conditional expectation.
\end{exam}

\hide{

\subsection{Topics}

Let $\cB\subseteq\cA$ be $*$-algebras and $p:\cA\to\cB$ be a bimodule projection. We say that an inducible q.m. $\cN\subseteq\cB$ is \emph{perfect} if
$$QM_\cA(\cN)=\Ind_{\cB\uparrow\cA}\cN.$$
A q.m. $\cM\subseteq\cA$ is called \emph{perfect} if $\Res\cM=\cM\cap\cB.$ 
In the following sections we will mainly study the following properties for $(\cA,\cB,p)$:
\begin{itemize}
  \item given an inducible q.m. $\cN\subseteq\cB$ is $\cN$ perfect?
  \item given q.m. $\cM\subseteq\cA$ is $\cM$ perfect?
  \item given a q.m. $\cN\subseteq\cB$ is $\cN=\Res\cM$ for some $\cM$?
  \item given a q.m. $\cM\subseteq\cA$ is $\cM=\Ind\cN$ for some q.m. $\cN\subseteq\cB$?
\end{itemize}
A characterization of group representations which are induced from a given subgroup is called an imprimitivity theorem. 
Thus, a general answer to the latter question can be viewed as an \emph{imprimitivity theorem for quadratic modules}. 
Every of these questions can be very complicated for a given $(\cA,\cB,p).$ We obtain satisfactory answers only for the case 
when $(\cA,\cB,p)=(L,K,\operatorname{tr}),$ where $L/K$ is a finite field extension and $\operatorname{tr}$ is the trace. 

}

\section{Averages of actions of finite groups}\label{subsect_aver}

Suppose that $G$ is a finite group which acts by $*$-automorphisms $\alpha_g,\ g\in G$ on a $*$-algebra $\cA,$ and let $\cB:=\{a\in\cA\ \mid\ \alpha_g(a)=a\ \mbox{for all}\ g\in G\}\subseteq\cA$ be the $*$-subalgebra of stable elements. Let $p:\cA\to\cB$ be the average mapping given by
\begin{align}\label{defpG}
p(a)=\frac{1}{|G|}\sum_{g\in G}\alpha_g(a),\ a\in\cA.
\end{align}
It can be easily shown that $p$ is a conditional expectation from $\cA$ onto $\cB,$ see also \cite{cks,ss}.

\begin{prop}\label{prop_aver}
Let $p$ be defined by (\ref{defpG}).
\begin{enumerate}
\item[(i)] If $\cN\subseteq\cB$ is an inducible quadratic module then $\Ind\cN$ is a $G$-invariant q.m. such that
$$\Ind\cN\cap\cB=\Res(\Ind\cN)$$
\item[(ii)] For every q.m. $\cM\subseteq\cA$ we have
$$\Res\cM\subseteq\Ind(\Res\cM).$$
\end{enumerate}
\end{prop}
\begin{proof}
(i): Note that $\cB\cap\Ind\cN=p(\cB\cap\Ind\cN)\subseteq p(\Ind\cN)=\Res(\Ind\cN).$ Let $a\in\Ind\cN.$ For arbitrary $x\in\cA$ we have
\begin{gather}
\nonumber p(x^*p(a)x)=\frac{1}{|G|}\sum_{g\in G}\alpha_g(x^*)p(a)\alpha_g(x)=\frac{1}{|G|^2}\sum_{g,h\in G}\alpha_g(x^*)\alpha_h(a)\alpha_g(x)=\\
\label{eq_aux2}=\frac{1}{|G|^2}\sum_{k\in G}\sum_{h\in G}\alpha_h(\alpha_k(x^*)a\alpha_k(x))=\frac{1}{|G|}\sum_{k\in G}p(\alpha_k(x^*)a\alpha_k(x))\in\cN,
\end{gather}
that is $p(a)\in\Ind\cN.$ It proves $\Res(\Ind\cN)\subseteq\cB\cap\Ind\cN.$ To see, that $\Ind\cN$ is invariant, take $a\in\Ind\cN.$ Then for every $x\in\cA$ and $g\in G$ we have
$p(x^*\alpha_g(a)x)=p(\alpha_{g^{-1}}(x^*)a\alpha_{g^{-1}}(x))\in\cN,$ hence $\alpha_g(a)\in\Ind\cN.$

(ii): Let $a\in\cM.$ For an arbitrary $x\in\cA$ calculation (\ref{eq_aux2}) yields
$$
p(x^*p(a)x)=\frac{1}{|G|}\sum_{k\in G}p(\alpha_k(x^*)a\alpha_k(x))\in\Res\cM.
$$
Hence $p(a)\in\Ind(\Res\cM),$ which implies the assertion.
\end{proof}

Assertion (i) suggests the following problem: Is every $G$-invariant q.m. $\cM\subseteq\cA$ induced? Cf. Proposition \ref{prop_prefect_pr_K} for a very special case of this.
In view of Subsection \ref{subsecres}, assertion (ii) implies that a q.m. $\cN \subseteq \cB$ is contained in $\ind \cN$ iff it is restricted.


\subsection{Matrix $\ast$-algebras}\label{secmatrix}
Let $\cB$ be a $\ast$-algebra, $N$ be a positive integer and $\cA=M_N(\cB)$ with involution $[b_{ij}]^*=[b_{ji}^*]$.
Let $G$ be the semidirect product $(\dZ/2)^N\rtimes_\phi\dZ/N$ where
$\phi(1) \colon (i_1,i_2,\ldots,i_N) \mapsto (i_N,i_1,\ldots,i_{N-1})$.
The action of $G$ on $\cA$ is defined by $A^g=T_g^* A T_g$ where, for $g=(i_1,\ldots,i_N,j)$,
$$T_g=\left[\begin{array}{ccccc} (-1)^{i_1} & 0 & \ldots  & 0 & 0 \\ 0 & (-1)^{i_2} & \ldots & 0 & 0 \\
\vdots & \vdots & \ddots & \vdots & \vdots\\ 0 & 0 & \ldots & (-1)^{i_{N-1}} & 0 \\ 0 & 0& \ldots & 0 & (-1)^{i_N} \end{array} \right]
\left[\begin{array}{ccccc} 0 & 1 & 0 & \ldots & 0\\ 0 & 0 & 1 & \ldots & 0\\ \vdots & \vdots & \vdots & \ddots & \vdots \\ 
0 & 0& 0& \ldots& 1 \\ 1 & 0 & 0 & \ldots & 0\end{array} \right]^j$$
If we identify $\cB \otimes I_N$ with $\cB$, then the average map $\gp(A)=\frac{1}{\vert G \vert} \sum_{g \in G} A^g$
coincides with the normalized trace $\tr(A)=\frac{1}{N}(a_{11}+\ldots+a_{NN})$.
Since $\tr (P^*A P) = \frac{1}{n} \sum_k(\sum_{i,j}p_{ik}^*a_{ij}^{}p_{jk}^{})$ for every $P=(p_{ij})$ and $A =(a_{ij})$
from $M_N(\cB)$, it follows that every q.m. $\cN$ in $\cB$ is inducible  and
\begin{gather*}
\Ind^\tr\cN=\{A\in M_N(\cB)_h\mid \sum_{i,j}x_i^*a_{ij}^{}x_j^{}\in\cN,\ \mbox{for all}\ x_1,\dots x_N\in\cB\}.
\end{gather*}

Let $\cB=\dR[x_1,\dots,x_n]$ and $S=\set{p_1,\dots,p_m}\subset\cB.$ Write
$K_S=\{a \in \RR^n \mid p_1(a) \ge 0, \ldots, p_m(a) \ge 0\}$ and  $\operatorname{Pos}(K_S)$
for the set of all polynomials from $\cB$ which are nonnegative on $K_S$. 
Let $T_S$ be the preordering in $\cB$ generated by $S$. Equivalently, 
$T_S$ is the quadratic module in $\cB$ generated by all square-free products of elements from $S$.
Let $\tr \colon M_N(\cB) \to \cB$ be the normalized trace. The following is clear:

\begin{prop}\label{posconj} With $S$ and $\cB$ as above,
$\Ind^{\tr} \operatorname{Pos}(K_S)$ coincides with the set of matrix polynomials from $M_N(\cB)_h$ which are positive semi-definite on $K_S.$
\end{prop}

The following remains unsolved:

\medskip

\begin{conj}
With $S$ and $\cA$  as above, $\Ind^{\tr} T_S= T_S \cdot \Sigma M_N(\cB)^2$,
i.e. $T_S$ is a perfect q.m.
\end{conj}

If $T_S$ is saturated (i.e. if $T_S=\operatorname{Pos}(K_S)$) then, by Proposition \ref{posconj}, 
the conjecture is equivalent to the following:
\emph{every matrix polynomial $P\in M_N(\cB)_h$ which is positive semi-definite on $K_S$ can be represented as 
$P=\sum\alpha_i A_i^*A_i,$ where $A_i\in  M_N(\cB)$ and $\alpha_i\in T_S.$}
The latter is known to be true in the following cases: 
\begin{itemize}
\item $\cB=\dR[x]$ and $S=\emptyset$, see \cite{jac} or \cite{dj}, 
\item $\cB=\dR[x]$ and $S=\set{x}$ or $\set{x,1-x}$, see \cite{ds} or \cite[Sec. 7]{ss2},
\item $\cB=\dR[x,y]$ and $S=\{1-x^2-y^2,x^2+y^2-1\}$, see \cite{ro}.
\end{itemize}

\subsection{Regular functions}
Let $V\subseteq\dR^n$ be a real algebraic variety and let $G$ be a finite automorphism group of $\dR[V].$ 
Consider also the corresponding action of $G$ on $V.$ Further, let $\cA=\dR[V],$ 
$\cB=\dR[V]^G$ be the subalgebra of $G$-stable elements, and let $p:\cA\to\cB$ be the canonical projection. 
Denote by $\cA_+$ the set of positive polynomials on $V$.
\begin{prop}\label{prop_indB_+}
$\Ind^p(\cA_+ \cap \cB)=\cA_+.$
\end{prop}
\begin{proof} 
The idea of the proof is to induce $*$-representations from $\cB$ to $\cA$ via $p,$ see Appendix. 
For $\chi\in V$ and $f\in\cA$ denote by $\Orb\chi\subseteq V$ the $G$-orbit of $\chi,$ by $[f]_\chi$ 
the corresponding function on $\Orb\chi$ and put $[\cA]_\chi=\set{[f]_\chi\ \mid\ f\in\cA}.$ 
It can be verified, that $\langle[f]_\chi,[g]_\chi\rangle:=(p(fg))(\chi)$ is a well-defined scalar product on $[\cA]_\chi.$ 
For each $f\in\cA$ introduce a linear operator $T_\chi(f)$ on $[\cA]_\chi$ via $T_\chi(f)[g]_\chi=[fg]_\chi,\ g\in\cA.$ 
Again, direct computations show that $T_\chi$ defines a $*$-representation of $\cA$ on the inner-product space $[\cA]_\chi.$ Then
$$
\langle T_\chi(f)[g]_\chi,[g]_\chi\rangle=\langle[fg]_\chi,[g]_\chi\rangle=p(f\cdot g^2)(\chi)\geq 0,\ \forall g\in\cA
$$
implies $f(\chi)\geq 0.$ Hence $p(f\cdot g^2)\in\cA_+ \cap \cB,\ \forall g\in\cA$ implies $f\in\cA_+.$
\end{proof}

The following conjecture remains open:
\medskip

\begin{conj}\label{conj2}
With $\cA$, $\cB$ and $p$ as above,
$\ind^p(\sum\cA^2\cap\cB)=\sum\cA^2.$
\end{conj}

\medskip
When $\sum\cA^2=\cA_+$, the conjecture follows from Proposition \ref{prop_indB_+}.
In particular, this is true if $\cA=\dR[x,y]/(x^2+y^2-1)\simeq\dR[S^1]$ by the Riesz-Fejer Theorem, see \cite[p. 117]{riesz}
or $\cA=\dR[x,y,z]/(x^2+y^2+z^2-1)\simeq\dR[S^2]$ by a result of Scheiderer, see \cite[Th. 3.2]{claus2}. For the action take
e.g. $G=\dZ/2\dZ$ with $(x,y)\mapsto(x,-y)$ and $(x,y,z)\mapsto(x,y,-z)$ respectively.

Let $\cA=\dR[x,y,z]/(x^2+y^2-1)\simeq\dR[S^1\times\dR^1]$ and let $\dZ/2\dZ$ act on $\cA$ by $(x,y,z)\mapsto(x,-y,z).$ 
It is not known yet whether $\sum\cA^2=\cA_+$. However, if the conjecture is true, this follows from 
Marshall's strip theorem \cite{mar1}. Namely, 
$\cB=\dR[x,y^2,z]/(x^2+y^2-1)\simeq\dR[x,z]$ and $\cA_+ \cap \cB=\operatorname{Pos}([-1,1]\times\dR^1)
= \sum \cB^2+(1-x^2) \sum \cB^2=\sum \cA^2 \cap \cB$ by \cite{mar1}, hence
$\cA_+=\Ind^p(\cA_+ \cap \cB)=\ind^p(\sum\cA^2\cap\cB)=\sum\cA^2$ by Proposition \ref{prop_indB_+}
and the conjecture.
\medskip

Proposition \ref{prop_indB_+} can be easily extended to compact groups (just replace sums with integrals over the Haar measure.)
On the other hand, Example \ref{czzb} will show that Conjecture \ref{conj2} fails for infinite compact groups.
(For convenience we will consider a complex $*$-algebra.)

\section{Group graded $\ast$-algebras}\label{subsect_grad}

Let $G$ be a discrete group and let $\cA$ be a $G$-graded $*$-algebra, that is $\cA=\bigoplus_{g \in G} \cA_g,$
where $\cA_g$ are linear subspaces such that
$$\cA_g+\cA_g \subseteq \cA_g, \quad -\cA_g \subseteq \cA_g, \quad \cA_g \cdot \cA_h \subseteq \cA_{gh}, \quad (\cA_g)^\ast \subseteq \cA_{g^{-1}}.$$
Write $\cB=\cA_\id$ where $\id\in G$ is the identity element. Then $\cB$ is a $*$-subalgebra of $\cA.$ Denote by $p$ the canonical projection of $\cA$ onto $\cB$, that is
\begin{gather}\label{eq_p_grad}
p:\cA \to \cB, \quad p(\sum_{g \in G} a_g)= a_\id,
\end{gather}
where $a_g\in\cA_g.$ It is easy to show that $p$ is a conditional expectation of $\cA$ onto $\cB,$ see e.g. \cite[Proposition 6]{ss}.

The elements of $\bigcup_{g \in G} \cA_g$ are called \emph{homogeneous}. For every homogeneous $a \in \cA$ we have $a^\ast \cB a \subseteq \cB$ and $p(a^\ast x a)=a^\ast p(x) a$ for every $x \in \cA$. Note also that the set $\sum \cA^2 \cap \cB$ consists of all finite sums of elements $a^\ast a$ with $a$ homogeneous. In particular, a q.m. $\cN\subseteq\cB$ is inducible if and only if $\cN$ contains all $a^\ast a$ with $a$ homogeneous.


\begin{prop}\label{prop_graded}
Let $p$ be defined by (\ref{eq_p_grad}).
\begin{enumerate}
\item[(i)] If $\cN\subseteq\cB$ is an inducible quadratic module then
$$\Ind\cN\cap\cB=\Res(\Ind\cN)$$
\item[(ii)] For every q.m. $\cM\subseteq\cA$ we have
$$\Res\cM\subseteq\Ind(\Res\cM).$$
\end{enumerate}
\end{prop}

\begin{proof}
(i): Note that $\cB\cap\Ind\cN=p(\cB\cap\Ind\cN)\subseteq p(\Ind\cN)=\Res(\Ind\cN).$ 
Let $b\in\Res(\Ind\cN)$, that is $b=p(a),\ a\in\Ind\cN.$ We show that $p(a)\in\Ind\cN.$ For let $x\in\cA,\ x=\sum_{g\in G}x_g,\ x_g\in\cA_g.$ Then
\begin{gather}\label{eq_aux}
p(x^*bx)=\sum_{g\in G}x_g^*bx_g^{}=\sum_{g\in G}x_g^*p(a)x_g^{}=\sum_{g\in G}p(x_g^*ax_g^{})
\end{gather}
and the latter is contained in $\cN$ since $a\in\Ind\cN.$

(ii): Let $b\in\Res\cM$ that is $b=p(a),\ a\in\cM.$ We show that $b\in\Ind(\Res\cM).$ For let $x\in\cA,\ x=\sum_{g\in G}x_g,\ x_g\in\cA_g.$ 
Then $x_g^*ax_g^{}\in\cM,\ \forall g\in G$ and calculation (\ref{eq_aux}) shows that $p(x^*bx)\in\Res\cM.$
\end{proof}


\begin{exam}\label{czzb}
Let $\cA=\dC[z,\overline{z}]$ be the $*$-algebra of complex polynomials with involution defined by $z\mapsto\overline{z}.$ We define a $\dZ$-grading on $\cA$ by putting $z\in\cA_1,\ \overline{z}\in\cA_{-1}.$ Then $\cB:=\cA_{0}=\dC[t],$ where $t=z\overline{z}.$ By \cite[Examples 7,13]{ss} there is a canonical action of $S^1\simeq\widehat{\dZ}$ on $\cA,$ namely $e^{\img\varphi}\in S^1,\ \varphi\in[0,2\pi]$ defines an automorphism, which maps $(z,\overline z)$ to $(e^{\img\varphi}z,e^{-\img\varphi}\overline z).$ The conditional expectation $\gp:\cA\to\cB$ defined by the $\dZ$-grading coincides with the average of the action of $S^1$ on $\cA:$
\begin{gather}\label{eq_aux1}
(\gp(f))(z,\overline z)=\int_0^{2\pi}f(e^{\img\varphi}z,e^{-\img\varphi}\overline z)d\varphi,
\end{gather}
see \cite[Examples 7,13]{ss}. Similarly to the Proposition \ref{prop_indB_+} it can be shown that $\Ind\cB_+=\cA_+.$ On the other hand $\cB_+$ consists of all polynomials positive on $\dR_+,$ hence
$$
\cB_+=\sum\dC[t]^2+t\dC[t]^2=\sum\cA^2\cap\cB=\Res\sum\cA^2.
$$
Thus we get $\Ind(\Res\sum\cA^2)=\cA_+$ which is known to be larger than $\sum\cA^2,$ see Motzkin's counterexample in \cite{mar}. Then Proposition \ref{prop_indres}, (1) implies that $\sum\cA^2\neq\Ind\cN$ for any q.m. $\cN\subseteq\cB$.
\end{exam}

\begin{exam}
The following $*$-algebra is a $q$-analogue of $\dC[z,\ov z]:$
$$\cA=\dC\langle x,x^*\mid\ xx^*=qx^*x\rangle,$$
where $q>0,\ q\neq 1$ is fixed. Positivity and sums of squares in $\cA$ were studied in \cite{css}.

$\cA$ has a natural $\dZ$-grading such that $x\in\cA_1,\ x^*\in\cA_{-1}.$ Then $\cB:=\cA_0$ is isomorphic to $\dC[t],$ where $t=x^*x.$ Let $\cB_+$ be the set of polynomials $f(t)\in\cB$ nonnegative on $\dR_+.$ Then $\sum\cA^2\cap\cB=\sum\cB^2+t\sum\cB^2=\cB_+.$ Denote by $\gp:\cA\to\cB$ the conditional expectation defined by the grading. The set of positive elements $\cA_+$ defined in \cite{css} has the following characterization. An element $f\in\cA_h$ is positive if and only if $\gp(y^*fy)\in\cB_+$ for all $y\in\cA.$ Equivalently, $\cA_+=\Ind\cB_+.$ Clearly we have $QM_\cA(\sum\cB^2+x^*x\sum\cB^2)=\sum\cA^2.$ By \cite[Theorem 2]{css} we have $\sum\cA^2\neq\cA_+,$ showing that $\cB_+$ is not perfect and that $\Ind(\Res\sum\cA^2)\neq \sum\cA^2,$ that is $\sum\cA^2$ is not induced.
\end{exam}

\begin{exam}\label{discex}
Let $G$ be a discrete group, $\cA=\dC[G]$ be its group $*$-algebra and $\cA_+$ be the set of elements which are positive in all unitary $*$-representations. Then $\cA$ is $G$-graded and $\cB$ can be identified with $\dC.$ Clearly, $\cN=\dR^+$ is the only quadratic module in $\cB$ and $QM_\cA(\cN)=\sum\cA^2.$ The $*$-representation $\pi_{reg}$ induced from $\cB=\dC$ to $\cA$ via $p$ is canonically isomorphic to the left regular representation of $\cA.$ By Proposition \ref{prop_ind_rep_mod} we have $\Ind\cN=QM_\cA(\pi_{reg}).$ Hence $\cN$ is perfect if and only if $QM_\cA(\pi_{reg})=\sum\cA^2.$

A well-known fact in the group theory is that every unitary representation of $G$ is weakly contained in $\pi_{reg}$ if and only if $G$ is amenable
(see e.g. \cite[Theorem 3.5.2]{greenleaf}). 
By Proposition \ref{prop_weakly}, if $G$ is not amenable then $QM_\cA(\pi_{reg})$ is strictly larger than $\cA_+\supseteq\sum\cA^2,$ so that $\cN$ 
is not perfect.

Let $G=\dZ^k.$ Then $\cA=\dC[\dZ^k]$ is isomorphic to the algebra $\dC[\dT^k]$ of polynomials on the $k$-dimensional torus. 
By the Riesz-Fejer Theorem (\cite[p. 117]{riesz}) and Scheiderer's theory (\cite[Prop. 6.1]{claus1} and \cite[Th. 3.2]{claus2}) 
we have $\sum\cA^2=\cA_+$ if and only if $k=1$ or $2.$ That is, $\cN$ is perfect if and only if $k=1$ or $2.$
\end{exam}

\section{Galois extensions}\label{sect_gal}

Let $(K,\ge)$ be an ordered field, $\bR$ its real closure and $\bC=\bR[i]$. Let $L$ be a finite extension of $K$
of degree $s$. We assume that both $L$ and $K$ have the trivial involution $\ast=\id$. Then the normalized trace 
$$
\tr_{L/K}=\frac{1}{s} \operatorname{tr}_{L/K}:L\to K
$$ 
is a bimodule projection from $L$ onto $K.$ 
By the Primitive Element Theorem, there exists $\theta \in C$ such that $L=K(\theta)$. 
Let $P(x)$ be the minimal polynomial of $\theta$ over $K$. Then $P(x)$ is irreducible, has degree $s$ and $L=K[x]/(P(x))$. 
Let $\theta_1,\ldots,\theta_s$ be all zeros of $P(x)$ in $\bC$. 
We may assume that $\theta_1,\ldots,\theta_r$ belong to $\bR$ and $\theta_{r+1},\ldots,\theta_s$ belong to $\bC \setminus \bR$.

Pick any $Q(x)\in K[x]$ which is not divisible by $P(x)$ and consider the quadratic form (from $K^s$ to $K$), see \cite[4.3.2]{bpr},
$$
\Her(P,Q) \colon (f_1,\ldots,f_s) \mapsto \tr_{L/K}(Q(\theta) (f_1+f_2 \theta+\ldots+f_s \theta^{s-1})^2).
$$
By the first part of \cite[Thm. 4.57]{bpr}, this form is nondegenerate (i.e. its rank is equal to $s$.)
By the second part of \cite[Thm. 4.57]{bpr}, the signature of $\Her(P,Q)$ is equal to $\operatorname{sign} Q(\theta_1)+\ldots+\operatorname{sign} Q(\theta_r)$. 
Inserting $Q=1$, we get that the signature of $\Her(P,1)$ is equal to $r$ (see also \cite[Thm. 4.58]{bpr}). We summarize the discussion above in

\begin{prop}\label{prop_ind_field_ext}
The following claims are equivalent:
\begin{enumerate}
\item the ordering $\ge$ is inducible (for the bimodule projection $\tr_{L/K}$, where $K$ and $L$ have identity involutions),
\item $\tr_{L/K}((f_1+f_2 \theta+\ldots+f_s \theta^{s-1})^2) \ge 0$ for all $f_1,\ldots,f_s \in K$,
\item the signature of $\Her(P,1)$ is equal to $s$,
\item $r=s$.
\end{enumerate}
If $\ge$ is inducible then for $a=Q(\theta)\in L$ the following claims are equivalent:
\begin{enumerate}
\item[(a)] $a \in \Ind(\ge)$,
\item[(b)] $\tr_{L/K}(Q(\theta)(f_1+f_2 \theta+\ldots+f_s \theta^{s-1})^2) \ge 0$ for all $f_1,\ldots,f_s \in K$,
\item[(c)] the signature of $\Her(P,Q)$ is equal to $s$,
\item[(d)] $Q(\theta_1)>0,\ldots,Q(\theta_s)>0$ (recall that $r=s)$,
\item[(e)] $a$ belongs to every ordering of $L$ extending $\ge$.
\end{enumerate}
\end{prop}
Recall that finite extensions of rational numbers $\dQ$ are called number fields. A number field $K$ is called \emph{totally real} 
if and only if $K=\dQ[\theta]$ and the minimal polynomial $p\in \dQ[t]$ of $\theta$ has only real roots.

\begin{cor}
Let $L$ be a number field. Then $\sum \dQ^2$ is inducible if and only if $L$ is totally real. Moreover, $\Ind(\sum\dQ^2)=\sum L^2$.
\end{cor}

For the rest of this section $L$ and $K$ are formally real and $L/K$ is a Galois extension with Galois group $G.$ 
The following proposition characterizes the preorderings on $L$ which are induced via $\tr_{L/K},$ thus it can be viewed 
as an \emph{Imprimitivity Theorem for preorderings on Galois extensions}. Recall that $\cM$ is induced if and only if $\cM=\Ind(\Res\cM).$
\begin{prop}\label{prop_field_impr}
Let $\cM\subseteq L$ be a proper preordering. Then $\cM$ is induced if and only if $\cM$ is $G$-invariant, that is $g(\cM)\subseteq\cM$ for all $g\in G.$
\end{prop}
\begin{proof}
If $\cM$ is induced, then by claim (i) of Proposition \ref{prop_aver} $\cM$ is $G$-invariant.

Conversely, let $\cM$ be $G$-invariant. By Proposition \ref{prop_indres} it is enough to show $\cM\supseteq\Ind(\Res\cM).$ 
We first consider the special case when $\cM=\rho^G,$ where $\rho$ is an ordering of $L$ and $\rho^G$ denotes $\cap_{g\in G}\,g(\rho).$
It is easily seen that $\rho\cap K=\rho^G\cap K=\Res\,\rho^G,$ and every ordering $g(\rho)$ extends $\rho\cap K.$ 
By the equivalence of claims (a) and (e) of Proposition \ref{prop_ind_field_ext}, the set $\Ind(\rho\cap K)$ is equal to the intersection of all orderings on $L$ containing $\rho\cap K,$ 
in particular $\Ind(\rho\cap K)\subseteq\rho^G.$ Thus we get $\Ind(\Res\,\rho^G)=\Ind(\rho\cap K)\subseteq\rho^G.$

Now let $\cM$ be an arbitrary $G$-invariant preordering. To show that $\Ind(\Res\cM)\subseteq\cM$ pick an element $a\in L\setminus\cM.$ 
Then $\cM$ can be extended to an ordering $\rho$ of $L$ such that $a\notin\rho\supseteq\rho^G=\Ind(\Res\rho^G).$ 
That is, there exists $x\in L$ such that $\tr_{L/K}(x^2\cdot a)\notin\Res\rho^G.$ 
Since $\cM$ is $G$-invariant, we have that $\cM\subseteq\rho^G,$ and so $\tr_{L/K}(x^2\cdot a)\notin\Res\cM,$ that is $a\notin \Ind(\Res\cM).$
\end{proof}

\begin{prop}\label{prop_prefect_pr_K}
Let $\cM\subset L$ be a preordering and $\cN=\Res\cM.$ Then $\cN$ is perfect, that is $\Ind\cN=QM_L(\cN)$. 
Moreover, for every $G$-invariant preordering $\cQ\subseteq L$ we have $\cQ=QM_L(\cQ\cap K).$
\end{prop}
\begin{proof}
Write $\cP=\Ind\cN.$ We show that $\cP=QM_L(\cN).$ Using Propositions \ref{prop_indres} and \ref{prop_aver} 
we have that $\cN=\Res(\Ind\cN)=\cP\cap K$ which implies the inclusion $QM_L(\cN)=QM_L(\cP\cap K)\subseteq\cP.$

Since $\cP\cap K$ is multiplicative, $QM_L(\cP\cap K)$ is a preordering. 
Hence, we can prove $QM_L(\cP\cap K)\supseteq\cP$ by showing that every ordering on $L$ which contains $\cP\cap K$ also contains $\cP.$ 
Assume to the contrary that $\rho$ is an ordering on $L$ which does not contain $\cP$ but contains $\cP\cap K$ and let $l\in\cP\setminus\rho.$ 
Then, since $\cP$ is $G$-invariant by Proposition \ref{prop_field_impr}, we have that $\tr_{L/K}(l\cdot\sum L^2)\subseteq\cP\cap\cK.$ 
By the equivalence of claims (b) and (e) of Proposition \ref{prop_ind_field_ext}, we have that $l$ belongs to every ordering of $L$
which contains $\cP\cap\cK$. In particular, we have that $l\in\rho$ which is a contradiction.

Let $\cQ$ be a $G$-invariant preordering, so that $\Res\cQ=\cQ\cap\cK.$ Then by Proposition \ref{prop_field_impr} we have
$$
\cQ=\Ind(\Res\cQ)=\Ind(\cQ\cap K)
$$
and since $\Res\cQ$ is perfect, the latter equals to $QM_L(\cQ\cap K).$
\end{proof}

\section{Cyclic algebras}\label{sect_csa}

Let $K$ be a field and let $\gA$ be a central simple algebra over $K$ with involution ${}^*.$ 
We define the \emph{normalized trace} of $\gA$ by 
$$
\tr_{\gA/K}=\frac{1}{n}\operatorname{trd}_{\gA/K}=\frac{1}{n^2}\operatorname{tr}_{\gA/K}:\gA\to K
$$
where $n^2=\dim_K\gA$ and $\operatorname{tr}_{\gA/K}$ (resp. $\operatorname{trd}_{\gA/K}$)
is the trace (resp. the reduced trace) of $\gA$. 
Note that $\tr_{\gA/K}$ is a bimodule projection.
An ordering $\rho$ of $K$ is called a \emph{$*$-ordering} 
if $\tr_{\gA/K}(a^*a)\in\rho$ for all $a\in\gA.$ Let the set of $*$-orderings be non-empty. 
An element $a=a^*$ is called \emph{positive} if $a\in\Ind\rho$ for all $*$-orderings $\rho$ of $K.$ 
The set of all positive elements in $\gA$ is denoted by $\gA_+.$ These notions were first introduced in \cite{ps}. 
Clearly, $*$-orderings are exactly inducible orderings, and the set of positive elements is $\Ind P_K,$ 
where $P_K$ is the intersection of all $*$-orderings on $K.$

We restrict the study of induced q.m.'s to the case when $\gA$ is a cyclic algebra $(L/K,\sigma,a)$ 
associated to a cyclic Galois extension $L/K$ of order $n>1.$ That is, there exist fixed elements 
$e\in\gA,\ a\in K^{\times}$ and a generator $\sigma$ of $\Gal(L/K)\simeq\cyc{n}$ such that
\begin{gather}\label{eq_gA_cyclic}
 \gA=\Un{}\cdot L\oplus e\cdot L\oplus\dots e^{n-1}\cdot L,\ e^n=a\cdot\Un{},\ \mbox{and}\\
 \nonumber l\cdot e=e\cdot\sigma(l),\ \mbox{for}\ l\in L.
\end{gather}
We identify $L$ with $L\cdot\Un{},$ and denote by $\gp_{\gA/L}:\gA\to L$ the canonical projection. 
It follows from the general theory, that $$\tr_{\gA/K}=\tr_{L/K}\circ\gp_{\gA/L}.$$ 

Recall, that $L$ is a splitting field of $\gA,$ that is $\gA\otimes_{K}L\simeq M_n(L).$ The action of $\cyc{n}$ on $\gA\otimes_{K}L$ is defined by $\sigma(x\otimes l)=x\otimes\sigma(l),\ x\in\gA,\ l\in L.$ Clearly, the stable subalgebra of $M_n(L)$ coincides with $\gA\otimes_K K\simeq\gA.$ Denote by $\gP$ the average mapping from $M_n(L)$ to $\gA.$ We describe the embedding $\gA\hookrightarrow M_n(L)$ and projection $\gP:M_n(L)\to\gA$ explicitly. For every $l\in L$ and $e$ define the matrices
\begin{gather}\label{eq_eps_matr_form}
\epsilon(l)=\sum_{i=1}^{n}E_{ii}\otimes\sigma^{i-1}(l),\ \epsilon(e)=E_{1,n}\otimes a+\sum_{i=1}^{n-1}E_{i,i+1}\otimes 1,
\end{gather}
where $E_{ij}\in M_n(L)$ is the matrix having $1$ at the $(i,j)$-place and $0$ elsewhere. Then
\begin{gather*}
\epsilon(l)=\left(
\begin{array}{ccccc}
    l   	&  0		&   \dots	&   		&  0  			\\
    0   	& \sigma(l) 	&   		&   		&  \vdots   		\\
    \vdots    	& 	 	& \sigma^2(l)  	&   		&     			\\
        	&   		&  		& \ddots  	&     			\\
        	&   		&   		&   		&  0  			\\
    0		& \dots		&   		&  0 		&    \sigma^{n-1}(l)
\end{array}
\right),\ \epsilon(e)=\left(
\begin{array}{ccccc}
  0 		& \dots		&   			& \dots	 & a      	\\
  1 		& 0 			&   			&   		 &      						\\
    		& 1				& \ddots 	&   		 &      						\\
    		&   			& \ddots 	&     	 &      						\\
    		&   			&   			& 			 & 							\\
    		&   			&   			& 		 1 & 0
\end{array}
\right)
\end{gather*}
It is easily seen that $\epsilon(e)^n=a\cdot\Id,\ \mbox{and} \ \epsilon(l)\cdot \epsilon(e)=\epsilon(e)\epsilon(\sigma(l)).$ 
Hence $\epsilon$ defines an embedding of $\gA$ into $M_n(L),$ and we can identify $\epsilon(\gA)$ with $\gA.$ 
Let $\gP:M_n(L)\to\gA$ be the additive mapping defined by
\begin{gather*}
\gP(E_{mk}\otimes l):=\frac{1}n e^{m-k}\sigma^{-k+1}(l),\ l\in L,\ m,k=1,\dots,n.
\end{gather*}
It can be checked by direct computations, that $\gP\circ\epsilon$ is identity on $\gA.$ It is also easily checked that $\gp_{\gA/L}\circ\gP:M_n(L)\to L$ coincides with $\gp_{11}$ defined by
$$
\gp_{11}(\sum_{i,j}E_{ij}\otimes l_{ij})=E_{11}\otimes l_{11}.
$$
Thus we have constructed a chain of embeddings:
\begin{gather}\label{eq_chain_emb}
K\subset L\subset\gA\subset M_n(L)
\end{gather}
and a tower of $K$-linear projections:
\begin{gather}\label{eq_tower_cexp}
M_n(L)\xrightarrow{\gP}\gA\xrightarrow{\gp_{\gA/L}} L\xrightarrow{\tr_{L/K}} K.
\end{gather}
We equip all algebras with involution so that every mapping in this chain becomes a conditional expectation.

Equations (\ref{eq_gA_cyclic}) imply that $\gA=(L/K,\sigma,a)$ is a $\cyc{n}$-graded algebra $\oplus_{k\in\dZ}\gA_k,$ where $\gA_k=e^k\cdot L.$ 
We define an involution on $\gA$ such that $\gA$ is a $\cyc{n}$-graded $*$-algebra, that is
\begin{gather}\label{eq_gA_grad_star}
\gA_k^*=\gA_{-k}^{}\ \mbox{for}\ k\in\cyc{n}.
\end{gather}
Cyclic algebras (more generally crossed-product algebras) satisfying (\ref{eq_gA_grad_star}) were studied in \cite{c} and \cite{ss2}. Then $\gA_0$ is invariant under the involution. 
We assume in addition, that involution is identity on $\gA_0\simeq L$ and consequently on $K\subseteq L.$ In particular, the involution is of the first kind \cite[Section 8]{scharlau}.
Further, we set $B=\diag{\lambda_0,\lambda_1,\dots,\lambda_{n-1}},$ where $\lambda_k:=e^{*k}e^k,\ k=1,\dots,n-1,\ \lambda_0=1$ and define an involution $\tau$ on $M_n(L)$ via $X^\tau:=B^{-1}X^TB,\ X\in M_n(L).$
\begin{prop}
Each embedding in (\ref{eq_chain_emb}) is a $*$-homomorphism, and each mapping in (\ref{eq_tower_cexp}) is a bimodule projection.
\end{prop}
\begin{proof}
By \cite[Proposition 9.4]{ss2}, $\gA$ is a $*$-subalgebra in $(M_n(L),\tau).$ For other inclusions the statement is trivial. By \cite[Proposition 9.6]{ss2}, the mapping $\gP$ is a bimodule projection. Further, since $\gp_{\gA/L}$ and $\tr_{L/K}$ are canonical projections coming from group grading and group action respectively, they are bimodule projections.
\end{proof}

An ordering $\rho$ of $L$ is called a $*$-\emph{ordering}, if $\gp_{\gA/L}(x^*x)\in\rho$ for all $x\in\gA$ or, equivalently, if $\rho$ is inducible to $\gA.$ Denote by $P_L$ and $P_K$ the intersection of all $*$-orderings on $L$ and $K$ respectively.
\begin{prop}\label{prop_star_ord}
${}$
\begin{enumerate}
\item[(i)] Assume there exist $*$-orderings on $L.$ Then $P_L$ is invariant under the action of $\cyc{n}.$
\item[(ii)] If $\rho$ is a $*$-ordering on $L,$ then $\rho\cap K$ is a $*$-ordering on $K.$
\item[(iii)] Let $\tau$ be a $*$-ordering on $K.$ Then $\tau$ is inducible to $L$ and $\Ind_{K\uparrow L}\tau$ is a preordering. 
Moreover, every ordering $\rho$ on $L$ containing $\Ind_{K\uparrow L}\tau$ is a $*$-ordering.
\item[(iv)] There exist $*$-orderings on $K$ if and only if there exist $*$-orderings on $L.$
\end{enumerate}
Assume that $*$-orderings exist. Then:
\begin{enumerate}
\item[(v)] $\tr_{L/K}(P_L)=P_L\cap K=P_K,$
\item[(vi)] $\Ind_{K\uparrow L}P_K=P_L,\ \Ind_{L\uparrow\gA}P_L=\Ind_{K\uparrow\gA}P_K=\gA_+.$
\end{enumerate}
\end{prop}
\begin{proof}
\noindent(i): Since $\gA$ is a $\cyc{n}$-graded $*$-algebra, an ordering $\rho$ of $L$ contains all $\gp_{\gA/L}(x^*x),\ x\in\gA$, if and only if $e^{*k}e^k\in\rho$ for all $k.$ 
In particular, $P_L$ is the smallest preordering containing all $e^{*k}e^k.$ Hence, it is enough to check that $\sigma(e^{*k}e^k)\in P_L$ for all $k.$

It follows from (\ref{eq_gA_cyclic}) that $(e^{*}e)\sigma(l)={e^{*}le}\ \mbox{for all}\ l\in L.$ Using this with $l=e^{*k}e^k$ we get
$$
\sigma(e^{*k}e^k)=e^{*(k+1)}e^{k+1}(e^*e)^{-1}=e^{*(k+1)}e^{k+1}\cdot e^*e\cdot((e^*e)^{-1})^2\in P_L.
$$

\smallskip
\noindent (ii): Let $\rho$ be a $*$-ordering and $x=\sum_{k=0}^{n-1}e^kl_k\in\gA.$ Since $\tr_{\gA/K}=\tr_{L/K}\circ\gp_{\gA/L}$ we get
\begin{gather}\label{eq_aux3}
\tr_{\gA/K}(x^*x)=\tr_{L/K}\left(\sum_{k=0}^{n-1}e^{*k}e^kl_k^2\right)=\sum_{k,m=0}^{n-1}\sigma^m(l_k^2)\sigma^m(e^{*k}e^k)
\end{gather}
which belongs to $P_L$ by (i). Hence $\tr_{\gA/K}(x^*x)\in P_L\cap K\subseteq\rho\cap K.$

\smallskip
\noindent (iii): Let $\tau$ be a $*$-ordering on $K$. Since $\sum L^2\subseteq\sum\gA^2$ and $\tr_{L/K}(\sum L^2)=\gp_{\gA/K}(\sum L^2)\subseteq\tau,$ $\tau$ is inducible to $L.$ 
By the equivalence of claims (a) and (e) of Proposition \ref{prop_ind_field_ext}, $\Ind_{K\uparrow L}\tau$ is a preordering. 
Further, we have $\tr_{L/K}(l^2e^{*k}e^k)\in\tr_{\gA/K}(\sum\gA^2)\subseteq\tau$ for all $k=1,\dots,n-1$ and $l\in L,$ 
that is $e^{*k}e^k\in\Ind_{K\uparrow L}\tau$ for all $k=1,\dots,n-1,$ which implies the last assertion.

\smallskip
\noindent (iv): Follows from (ii) and (iii).

\smallskip
\noindent (v): $\tr_{L/K}(P_L)=P_L\cap K$ follows from $\sigma$-invariance of $P_L.$ Calculation (\ref{eq_aux3}) implies the inclusion $P_K\subseteq P_L.$ Let $\tau$ be a $*$-ordering on $K.$ Then by (iii) $\Ind_{K\uparrow L}\tau$ contains $P_L.$ On the other hand Proposition \ref{prop_indres}(ii) implies $\tr_{L/K}(\Ind_{K\uparrow L}\tau)\subseteq\tau,$ hence $\tr_{L/K}(P_L)\subseteq\tau.$ That is $\tr_{L/K}(P_L)$ is contained in all $*$-orderings on $K,$ i.e. $\tr_{L/K}(P_L)\subseteq P_K.$

\smallskip
\noindent (vi): First equation follows from (v) and Proposition \ref{prop_field_impr}. The second equation follows then by induction in stages (Proposition \ref{prop_ind_stage}).
\end{proof}

\begin{prop}
$P_L$ and $P_K$ are perfect quadratic modules in $\gA$, that is
$$QM_\cA(P_L)=\Ind_{L\uparrow\gA}(P_L)=\gA_+=\Ind_{K\uparrow\gA}(P_K)=QM_\cA(P_K).$$
Moreover $\gA_+$ is a non-commutative preordering.
\end{prop}
\begin{proof}
The first equality follows directly from \cite[Proposition 9.10]{ss2}. By Proposition \ref{prop_prefect_pr_K} we have $\Ind_{K\uparrow L}P_K=QM_L(P_K).$ Hence
$$
QM_\cA(P_K)=QM_\cA(QM_L(P_K))=QM_\cA(\Ind_{K\uparrow L}(P_K))=QM_\cA(P_L).
$$
The fact that $\gA_+$ is a non-commutative preordering is the statement of \cite[Proposition 9.9]{ss2}.
\end{proof}


\section{Weyl algebra}\label{sect_weyl}

In this section we study the $*$-algebra
\begin{gather}\label{eq_aa*_a*a}
\cA=\dC\langle a, a^\ast \mid a a^\ast-a^\ast a=1\rangle,
\end{gather}
which is called the \emph{Weyl algebra.} Using the linear substitutions 
$$X=\dfrac{a+a^\ast}{\sqrt{2}}=q,\ Y=\dfrac{a-a^\ast}{\sqrt{2}}=ip$$
we get another two presentations for $\cA:$
\begin{gather}\label{eq_YX_XY}
\cA=\dC\langle X,Y \mid Y X-X Y=1, \quad X^\ast=X, \quad Y^\ast=-Y\rangle \\
= \dC\langle p,q \mid p q-q p=-i, \quad p^\ast=p, \quad q^\ast=q\rangle
\end{gather}

Denote by $N$ the \emph{number operator}
$$
N=a^*a=\frac 12(X^2-Y^2-1)=\frac 12(p^2+q^2-1).
$$

The standard $*$-representation of $\cA$ used in quantum mechanics is the Schr\"odinger representation $\pi_S$ of (\ref{eq_YX_XY}) 
which acts on the Schwarz space $\cS(\dR)=\cD(\pi_S)$ as follows
$$
(\pi_S(X)\varphi)(t)=t\varphi(t),\ (\pi_S(Y)\varphi)(t)=\frac{d}{dt}\varphi(t),
$$
see \cite[Section VIII.5]{rs} for details. $\pi_S$ is unitarily equivalent to the Fock-Bargmann representation $\pi_F$ 
of (\ref{eq_aa*_a*a}) which is defined as follows. 
The domain $\cD(\pi_F)$ consists of all sequences\footnote{Here, $\dN_0=\{0,1,2,\ldots\}$ is the set of all nonnegative integers.}
$\varphi=(\varphi_n)_{n\in\dN_0}\in\ell^2(\dN_0)$ such that 
$\sum_k k^M|\varphi_k|^2<\infty$ for all $M\in\dN_0;$ $\pi_F$ acts on the orthonormal base $e_n,\ n\in\dN_0$, of $\ell^2(\dN_0)$ as follows
\begin{gather}\label{eq_fock}
\pi_F(a)e_k=\sqrt k\,e_{k-1},\ \pi_F(a^*)e_k=\sqrt{k+1}\,e_{k+1},\ k\in\dN_0,
\end{gather}
where, $e_{-1}=0.$ We denote by $\cA_+$ the quadratic module $QM_\cA(\pi_F);$ the elements of $\cA_+$ are called \emph{positive.} Since the linear set $\cD_0(\pi_F)=\Lin\set{e_k,\ k\in\dN_0}$ is dense in $\cD(\pi_F)$ in the graph topology, we get
\begin{gather}\label{eq_fock_pos}
a\in\cA_+\Longleftrightarrow \langle\pi_F(a)\varphi,\varphi\rangle\geq 0,\ \forall\varphi\in\cD_0(\pi_F).
\end{gather}
\medskip

We will consider the following $\dZ$-grading of $\cA$
\begin{gather}\label{eq_grading_weyl}
\cA=\bigoplus_{k \in \dZ} e_k \cB
\end{gather}
where $\cB=\dC[N]$ and
$$
e_k=\left\{ \begin{array}{cc} a^k, & k \ge 0, \\ (a^\ast)^{-k}, & k <0. \end{array} \right.
$$
Note that for every $k \in \dZ$, $N e_k=e_k (N-k)$, so that $f(N) e_k =e_k f(N-k)$ for every $f(N) \in \cB$ and every $k \in \dZ$. It implies that $(e_k \cB)(e_l \cB) \subseteq e_{k+l}\cB$ and $(e_k \cB)^\ast=e_{-k} \cB$.\medskip

The following fact was first proved in \cite{fs}.
\begin{lemma}\label{lemma_sosweyl1}
$\sum\cA^2\cap\cB$ is the q.m. in $\cB$ generated by
$$
\set{N(N-1) \cdots (N-k)\mid k \in\dN_0}
$$
\end{lemma}
\begin{proof}
It follows from (\ref{eq_grading_weyl}) that q.m. $\sum \cA^2 \cap \cB$ is the smallest q.m. in $\cB$ which contains
\begin{gather*}
e_k^\ast e_k = \left\{ \begin{array}{cc} N(N-1) \cdots (N-k+1), & k > 0, \\ 1, & k=0, \\
(N+1)(N+2) \cdots (N+(-k)), & k <0. \end{array} \right.
\end{gather*}
It is easily seen that $e_k^\ast e_k,\ k\leq 0$ is contained in $\sum\cB^2+N\sum\cB^2.$ Hence $\sum\cA^2\cap\cB$ is generated by $e_k^\ast e_k,\ k>0.$
\end{proof}

Clearly, 
$$\operatorname{Pos}(\dN_0):=\set{f\in\dC[N]\mid f(k)\geq 0,\ \forall k\in\dN_0}$$
is a q.m. in $\cB=\dC[N]$. It follows from (\ref{eq_fock}) that 
$\pi_F(N)e_k=ke_k,\ k\in\dN_0,$ hence $f(N)\in\cB$ is in $\cA_+$
if and only if $f(k)\geq 0$ for all $k\in\dN_0$. Therefore:

\begin{lemma}\label{lemma_sosweyl2}
$\cA_+\cap\cB=\operatorname{Pos}(\dN_0).$
\end{lemma}

It follows from either Lemma \ref{lemma_sosweyl1} or Lemma \ref{lemma_sosweyl2} that the q.m. $\operatorname{Pos}(\dN_0)$ is inducible.
Moreover, we have the following:

\begin{prop}\label{prop_fock_qm}
$\cA_+=\Ind \operatorname{Pos}(\dN_0)$
\end{prop}
\begin{proof}
For $k\in\dN_0$ let $\rho_k$ denote the one-dimensional $*$-representation of $\cB$ defined by $\rho_k(N)=k$.
Then $QM_\cB(\rho_k)=\set{f\in\dC[N]\mid f(k)\geq 0}$. 
As shown in \cite[Example 1]{ss}, the $*$-representation 
$\Ind_{\cB\uparrow\cA}\rho_k$ is unitarily equivalent to the Fock-Bargmann representation restricted onto $\cD_0(\pi_F)$. 
Proposition \ref{prop_ind_rep_mod} implies $\cA_+=QM_\cA(\Ind\rho_k)=\Ind (QM_\cB(\rho_k)).$
Using the last equation we get for $x\in\cA_h$
\begin{gather*}
x\in\cA_+\Longleftrightarrow\forall y\in\cA\forall k\in\dN_0\ p(y^*xy)\in QM_\cB(\rho_k)\Longleftrightarrow\\
\Longleftrightarrow \forall y\in\cA\ p(y^*xy)\in\cap_{k\in\dN_0}QM_\cB(\rho_k)\Longleftrightarrow x\in\Ind \operatorname{Pos}(\dN_0),
\end{gather*}
which completes the proof.
\end{proof}

Let $\varphi_0$ be the linear functional on $\cA$ defined by $\varphi_0(a^{*m}a^n)=0$ for $m+n>0$ and $\varphi_0(\Un{})=1.$ The following proposition gives another useful characterization of $\cA_+.$
\begin{prop}
An element $x\in\cA_h$ belongs to $\cA_+$ if and only if $\varphi_0(y^*xy)\geq 0$ for all $y\in\cA.$
\end{prop}
\begin{proof}
Let $p:\cA\to\cB$ denote the canonical conditional expectation defined by the grading (\ref{eq_grading_weyl}). Then 
$\varphi_0=\rho_0\circ p,$ where $\rho_0:\cB\to\dC$ is the character on $\cB$ defined by $\rho_0(f(N))=f(0).$ Obviously we have $\Ind_{\dC\uparrow\cB}^{\rho_0}\dR_+=\cN_0.$ Using induction in stages (Proposition \ref{prop_ind_stage}) and Proposition \ref{prop_fock_qm} we get
\begin{gather*}
\set{x\in\cA_h\mid \forall y\in\cA\ \varphi_0(y^*xy)\geq 0}=\Ind^{\varphi_0}\dR_+=\\
=\Ind^p(\Ind^{\rho_0}\dR_+)=\Ind^p\cN_0=\cA_+.
\end{gather*}
\end{proof}

\noindent The following Proposition was stated in \cite{wor}, see also \cite{fs}.
\begin{prop}\label{weylsos}
$(N-1)(N-2)\in(\cA_+\cap\cB)\setminus(\sum\cA^2\cap\cB).$
\end{prop}
\begin{proof}
Apply  Lemmas \ref{lemma_sosweyl1} and \ref{lemma_sosweyl2}.
\end{proof}

It follows from Proposition \ref{weylsos} that $QM_\cA(\cA_+ \cap \cB)$ is strictly larger than 
$\sum\cA^2.$ In view of this fact, we consider the following problem:

\section{Does the Weyl algebra satisfy (Q3)?}\label{secweylexam}

Our aim now is to show that the answer to this question is negative, i.e. $\cN':=QM_\cA(\cA_+\cap \cB)=QM_\cA(\operatorname{Pos}(\dN_0))$ 
is a proper subset of $\cA_+=\Ind \operatorname{Pos}(\dN_0).$ In other words, $\operatorname{Pos}(\dN_0)$ is not perfect.
Motivated by the commutative Motzkin polynomial, we consider the elements
\[
L_K:=Y^2 X^2 Y^2+(-Y)(X^4-K X^2)Y.
\]
for $K \in \RR$. 
Our aim is to prove the following.

\begin{prop}
$L_5+1.4$ belongs to $\cA_+\setminus \cN'$.
\end{prop}

To prove that $L_5+1.4$ belongs to $\cA_+$ it suffices in view of Proposition \ref{prop_fock_qm} to show
that $L_5+1.4$ is positive semi-definite in the Fock-Bargmann representation. Since the element $1+X^2$ is
invertible in the Schr\"odinger representation, it suffices to observe that
\[
(1+X^2)(L_5+1.4)(1+X^2)=(v_1^\ast)^T A_1 v_1+(v_2^\ast)^T A_2 v_2
\]
where
\[
A_1=
\left[
\begin{array}{cc}
 \frac{253}{100} & \frac{121}{100} \\[2mm]
 \frac{121}{100} & \frac{29}{50}
\end{array}
\right]
\ge 0
,\quad
v_1=
\left[
\begin{array}{c}
X^3 Y \\[2mm] X^2 Y^2
\end{array}
\right]
\]
and
\[
A_2=
\left[
\begin{array}{ccccc}
 \frac{251}{25} & \frac{1491}{100} & \frac{911}{50} & \frac{27}{10} & \frac{1537}{500} \\[2mm]
 \frac{1491}{100} & \frac{4657}{200} & \frac{1357}{50} & \frac{3711}{1000} & \frac{951}{200} \\[2mm]
 \frac{911}{50} & \frac{1357}{50} & \frac{1681}{50} & \frac{26}{5} & \frac{549}{100} \\[2mm]
 \frac{27}{10} & \frac{3711}{1000} & \frac{26}{5} & 1 & \frac{71}{100}\\[2mm]
 \frac{1537}{500} & \frac{951}{200} & \frac{549}{100} & \frac{71}{100} & 1
\end{array}
\right]
\ge 0,
\quad
v_2=
\left[
\begin{array}{c}
X \\[2mm] X^3 \\[2mm] X^2 Y \\[2mm] X Y^2 \\[2mm] X^4 Y + X^3 Y^2
\end{array}
\right].
\]

We will divide the proof that $L_5+1.4 \not\in \cN'$ into several steps. The first step is:

\subsection*{Step 1:}
If $L_5+\lambda \in \cN'$ for some $\lambda \in \dR$ then
\[
L_5+\lambda=(u^\ast)^T B u + \alpha N^2+\beta N+\gamma
\]
where
\[
u=\left[ \begin{array}{cccccccc} 1 & X & Y &  X^2 & X Y & Y^2 & X^2 Y & X Y^2 \end{array} \right]^T,
\]
$B$ is a positive semi-definite hermitian $8 \times 8$ matrix
and $\alpha,\beta,\gamma \in \RR$ are such that
either $\alpha>0$ or $\alpha=\beta=\gamma=0$.

\begin{proof}
This is a story of two monomial orderings in the $pq$ representation:
\[
1 \prec_1 p \prec_1 q \prec_1 p^2 \prec_1 pq \prec_1 q^2 \prec_1 p^3 \prec_1 p^2 q \prec_1 p q^2 \prec_1 q^3 \prec_1 \ldots
\]
and
\[
1 \prec_2 q \prec_2 p \prec_2 q^2 \prec_2 pq \prec_2 p^2 \prec_2 q^3 \prec_2 p q^2 \prec_2 p^2 q \prec_2 p^3 \prec_2 \ldots
\]
Let $\lc_1$ and $v_1$ (resp. $\lc_2$ and $v_2$) be the leading coefficient and the leading multidegree w.r.t.
$\prec_1$ (resp. $\prec_2$). For example, $v_1(N)=(0,2)$, $v_2(N)=(2,0)$ and $\lc_1(N)=\lc_2(N)=\frac12$.
Similarly, $v_1(L_5)= (2,4)$, $v_2(L_5)=(4,2)$ and $\lc_1(L_5)=\lc_2(L_5)=1$.

\begin{comment}
The basic properties of
$\lc_i$ and $v_i$, $i=1,2$, are
(1) $\lc_i(f g)=\lc_i(f) \lc_i(g)$ and $v_i(f g)= v_i(f)+v_i(g)$  for all $f,g \in \cA$.
(2) $\lc_i(f^\ast)=\overline{\lc_i(f)}$ and $v_i(f^\ast)=v_i(f)$ for all $f \in \cA$.
(3) For every $f,g \in \cA$ such that $v_i(f)>v_i(g)$ we have that $\lc_i(f+g)=\lc(f)$
and $v_i(f+g)=v_i(f)$.
(4) For every $f,g \in \cA$ such that $v_i(f)=v_i(g)$ and $\lc_i(f) \ne -\lc_i(g)$ we have that
$\lc_i(f+g)=\lc_i(f)+\lc_i(g)$ and $v_i(f+g)=v_i(f)$.
It follows that for every $f \in \A$ with $\lc_i(f)>0$, $i=1,2$, and every nonzero $g \in \A$
we also have $\lc(g^\ast f g)=\lc_i(f) \vert \lc_i(g) \vert^2>0$ for $i=1,2$.
\end{comment}

If $L_5+\lambda$ belongs to $\cN'$, then there exist nonzero elements $f_1,\ldots,f_k \in \operatorname{Pos}(\dN_0)$
and nonzero elements $g_{ij} \in \cA$ such that $L_5+\lambda=\sum_{i,j} g_{ij}^\ast f_i g_{ij}$.
Note that $\lc_1(f_i)=\lc_2(f_i)>0$ for all $i=1,\ldots,k.$ 
It follows that
$\lc_t(g_{ij}^\ast f_i g_{ij})=\lc_t(f_i) \vert \lc_i(g_{ij}) \vert^2>0$ for $t=1,2$ and all $i,j$.
Therefore, $v_t(\sum_{i,j} g_{ij}^\ast f_i g_{ij})=\max_{i,j} v_t(g_{ij}^\ast f_i g_{ij})$
for $t=1,2$ and all $i,j$. In particular, $v_t(f_i)+2 v_t(g_{ij}) \le v_t(L_5)$.
Now, we have to consider several cases:
\begin{itemize}
\item If $\deg_N f_i=0$, then $v_1(f)=v_2(f)=(0,0)$. If follows that $v_1(g_{ij}) \le (1,2)$ and $v_2(g_{ij}) \le (2,1)$.
Therefore $g_{ij}$ is in the span of $\{1,p,q,p^2,pq,q^2,pq^2,p^2q\}$.
\item If $\deg_N f_i=1$, then $v_1(f_i)=(0,2)$ and $v_2(f_i)=(2,0)$. If follows that $v_1(g_{ij}) \le (1,1)$ and
$v_2(g_{ij}) \le (1,1)$. Therefore $g_{ij}$ is in the span of $\{1,p,q,pq\}$.
\item If $\deg_N f_i=2$, then $v_1(f_i)=(0,4)$ and $v_2(f_i)=(4,0)$. If follows that $v_1(g_{ij}) \le (1,0)$ and
$v_2(g_{ij}) \le (0,1)$. Therefore $g_{ij}$ is in the span of $\{1\}$.
\item If $\deg_N f_i \ge 3$, then we have no $g_{ij}$, so this case is not possible.
\end{itemize}
Note that every $f_i \in \operatorname{Pos}(\dN_0)$ with $\deg_N f_i=1$ is of the form
$f_i=\sigma_i+\tau_i N$ where $\tau_i>0$ and $\sigma_i \ge 0$. Finally combining the cases
$\deg_N f_i=0$ and $\deg_N f_i=1$ we get the term $(u^\ast)^T B u $ and the case
$\deg_N f_i=2$ gives the term $\alpha N^2+\beta N+\gamma$ with $\alpha>0$.
\end{proof}

\subsection*{Step 2:}
If $L_5+\lambda \in \cN'$ for some $\lambda \in \RR$ then
\[
L_5+\lambda=(v^\ast)^T C v
\]
where
\[
v=\left[ \begin{array}{cccccc} 1 & X & Y & X Y & X^2 Y & X Y^2 \end{array} \right]^T,
\]
and $C$ is a positive semi-definite hermitian $6 \times 6$ matrix.

\begin{proof}
By Step 1, we can write $L_5+\lambda=(u^\ast)^T B u+ \alpha N^2+\beta N+\gamma$ where
$u=\left[ \begin{array}{cccccccc} 1 & X & Y &  X^2 & X Y & Y^2 & X^2 Y & X Y^2 \end{array} \right]^T$
and either $\alpha>0$ or $\alpha=\beta=\gamma=0$. Expanding both sides and comparing coefficients,
we get several equations. The most interesting are $b_{44}+\frac{\alpha}{4}=0$ and $b_{66}+\frac{\alpha}{4}=0$.
Since $B$ is positive semi-definite and $\alpha \ge 0$, it follows that $\alpha=b_{44}=b_{66}=0$.
Therefore the $4$th and the $6$th row and column of $B$ are zero and by assumption, also $\beta=\gamma=0$.
\end{proof}

\subsection*{Step 3:}
If $L_5+\lambda \in \cN'$ for some $\lambda \in \dR$ then $\lambda \ge \frac32$.

\begin{proof}
Expanding the right hand side of the equation from Step 2 and comparing the coefficients, we get 16 equations.
Let us write down six of them:

\begin{equation}
\tag{$1$}
\lambda-c_{11}+c_{32}+c_{41}-2 c_{62}=0,
\end{equation}
\begin{equation}
\tag{$Y^2$}
c_{33}+c_{36}-2c_{63}-2 c_{66}=-2,
\end{equation}
\begin{equation}
\tag{$XY$}
-c_{14}-c_{23}+c_{32}+2c_{35}+c_{41}+2c_{44}+2c_{53}-4c_{62}-6c_{65}=-10,
\end{equation}
\begin{equation}
\tag{$X^2$}
3c_{52}-c_{22}=0,
\end{equation}
\begin{equation}
\tag{$X^2 Y^4$}
c_{66}=1,
\end{equation}
\begin{equation}
\tag{$X^4 Y^2$}
c_{55}=1.
\end{equation}
Multiplying the equations by $1,-\frac32,-\frac{3}{8},\frac{1}{6},-\frac{33}{8}$ and $-\frac{9}{8}$ respectively
and adding them, we get
\[
\lambda-\Delta=\frac32,
\]
where
$
\Delta=c_{11}-\frac{3}{8} c_{14}+\frac{1}{6} c_{22}-\frac{3}{8} c_{23}-\frac{5}{8} c_{32}+\frac{3}{2} c_{33}
+\frac{3}{4} c_{35}+\frac{3}{2} c_{36}-\frac{5}{8} c_{41}+\frac{3}{4} c_{44}-\frac{1}{2} c_{52}+\frac{3}{4} c_{53}
+\frac{9}{8} c_{55}+\frac{1}{2} c_{62}-3 c_{63}-\frac{9}{4} c_{65}+\frac{9}{8} c_{66}.
$
We can write
\[
\Delta=\tr(A C^T),
\]
where
\[
A=
\left[
\begin{array}{cccccc}
 1 & 0 & 0 & -\frac{3}{8} & 0 & 0 \\[2mm]
 0 & \frac{1}{6} & -\frac{3}{8} & 0 & 0 & 0 \\[2mm]
 0 & -\frac{5}{8} & \frac{3}{2} & 0 & \frac{3}{4} & \frac{3}{2} \\[2mm]
 -\frac{5}{8} & 0 & 0 & \frac{3}{4} & 0 & 0 \\[2mm]
 0 & -\frac{1}{2} & \frac{3}{4} & 0 & \frac{9}{8} & 0 \\[2mm]
 0 & \frac{1}{2} & -3 & 0 & -\frac{9}{4} & \frac{9}{8}
\end{array}
\right]
\]
The matrix $A$ is not hermitian but since $\Delta \in \RR$ and $C$ is hermitian,
we have by the basic properties of trace that
\[
\Delta=\tr(\frac12(A^T+\bar{A})C).
\]
To finish the proof, note that $\frac12(A^T+\bar{A})$ and $C$ are both hermitian and positive semi-definite,
hence $\Delta \ge 0$.
\end{proof}

\hide{
The identity
\[
L_K+\frac{1}{8} (K-3) (K-1)=h^\ast h+o
\]
where
\[
h =X^2 Y+X Y^2+\frac{(K-3) (K-1) X}{4 (K-5)}-\frac{(K-9) (K-3) Y}{4 (K-5)}
\]
and
\[
o=-\frac{(K-9) (K-7) (K-3) (K-1) N}{8 (K-5)^2}.
\]
shows that for $K=1,3,5,7$, $L_K+\frac{1}{8} (K-3) (K-1) \in \sum \cA^2$.
Our numerical experiments also show that in these four cases the element
$L_K+\frac{1}{8} (K-3) (K-1)-\eps \cdot 1$
is not in $\ind(N_0)$ for any $\eps>0$. It is tempting to ask what happens if $K=5$.
}

We have the following identity
\[
L_5+\frac32=h_1^\ast h_1+h_2^\ast h_2,
\]
where
\[
h_1=\frac12 (3X+Y)=\frac{1}{\sqrt{2}}(2 a+a^\ast)
\]
and
\[
h_2=\frac12(3 X + Y + 2 X^2 Y + 2 X Y^2)=\frac{1}{\sqrt{2}}(a^3 - (a^\ast)^2 a).
\]

\section{Appendix : Rigged bimodules}

In this section we briefly discuss a generalization of conditional expectations and of the induction of $*$-representations and quadratic modules.

Let $\bR$ be a real closed field, $\bC=\bR[i]$ and let $\cA$ and $\cB$ be unital $*$-algebras over $\bK$ where $\bK=\bR$ or $\bC.$
Let $\gX$ be a $\bK$-vector space equipped with a structure of a left-$\cA$-right-$\cB$-bimodule, such that $\Un{\cA}x=x\Un{\cB}=x,\ \forall x\in\gX.$ Suppose there exists a map $\langle{\cdot}, {\cdot}\rangle_\cB : \gX \times \gX \to \cB$ which satisfies the following conditions for $x,y\in \gX$, $b \in \cB$, and $a \in \cA$:
\begin{itemize}
\item[(RM1)]\quad $\langle x,y\rangle_{\cB}^*=\langle y,x \rangle_{\cB}$,
\item[(RM2)]\quad $\langle xb,y\rangle_{\cB}=\langle x,y\rangle_{\cB}b$,
\item[(RM3)]\quad $\langle ax, x \rangle_{\cB} = \langle x, a^\ast x \rangle_{\cB}$,
\end{itemize}
then $\gX$ is called a \textit{right $\cB$-rigged left $\cA$-module}. Note, that (RM1) and (RM2) imply that
$\langle x,yb\rangle_{\cB}=b^*\langle x,y\rangle_{\cB},\ \mbox{for}\ x,y\in\gX,\ b\in\cB.$

For a q.m. $\cM\subseteq\cA$, define
\begin{gather*}
\Res^\gX\cM:=\set{\sum_i\langle a_i x_i,x_i\rangle_\cB,\ a_i\in\cM,\ x_i\in\gX}.
\end{gather*}
For a q.m. $\cN\subseteq\cB$ define
\begin{gather*}
\Ind^\gX\cN:=\{a\in\cA_h\ \mid\ \langle ax, x \rangle_{\cB}\in\cN\ \mbox{for all}\ x\in\gX\}.
\end{gather*}

\noindent The set $\Res\cM=\Res^\gX\cM$ is called \textit{restriction of $\cM$ onto $\cB$ via $\gX$}. The set $\Ind\cN=\Ind^\gX\cN$ is called \textit{induction of $\cN$ from $\cB$ to $\cA$ via $\gX.$} The proof of the following proposition goes by direct computations.

\begin{prop}
Let $\cN\subseteq\cB,\ \cM\subseteq\cA$ be quadratic modules. Then $\Res\cM$ is a q.m. if and only if
\begin{gather}\label{eq_restrcrig}
\Un{\cB}=\sum_{i=1}^m\langle x_i,x_i \rangle_{\cB},\ \mbox{for some}\ x_i\in\gX.
\end{gather}

\noindent The set $\Ind\cN$ is a q.m. if and only if
\begin{gather}\label{eq_indcbrig}
\set{\langle x,x\rangle_\cB,\ x\in\gX}\subseteq\cN.
\end{gather}

\noindent Moreover, if (\ref{eq_indcbrig}) is satisfied, then $\Ind\cN$ is proper if and only if
\begin{gather*}
\set{\langle x,x\rangle_\cB,\ x\in\gX}\not\subseteq\supp\cN=\cN\cap(-\cN).
\end{gather*}
\end{prop}
\noindent We say that a q.m. $\cN\subseteq\cB$ is \emph{inducible via $\gX$} if (\ref{eq_indcbrig}) is satisfied. 
The analogue of Proposition \ref{prop_indres} can be also proved in this context.

\begin{exam}
Let $\cB\subseteq\cA$ be $*$-algebras such that $\Un{\cA}=\Un{\cB}$ and let $p:\cA\to\cB$ be a bimodule projection. 
Then $\gX=\cA$ has a structure of an $\cA$-$\cB$-bimodule defined by left and right multiplication respectively. 
For $x,y\in\cA$ put $\langle x, y\rangle_\cB:=p(y^*x).$ It is easily verified, that $\gX$ becomes a right $\cB$-rigged left $\cA$-module. 
The functors $\Res^\gX$ and $\Ind^\gX$ coincide with $\Res^p$ and $\Ind^p$ respectively. The notion of inducibility via $\gX$ coincides with inducibility via $p.$
\end{exam}

\begin{exam}
Let $\cA$ be a $*$-algebra over $\dC$ and let $\pi$ be a $*$-re\-pre\-sen\-ta\-ti\-on of $\cA$ on an inner-product space $\cV.$ 
Then $\cV$ together with $\dC$-valued inner-product $\langle\cdot,\cdot\rangle$ becomes a right $\dC$-rigged left $\cA$-module, 
where $a\cdot v=\pi(a)v$ and $v\cdot\lambda=\lambda v,$ for $a\in\cA,\ v\in\cV,\ \lambda\in\dC.$ By definition of $\Ind^\gX$ we have
$$
\Ind^\gX\dR_+=QM_\cA(\pi).
$$
\end{exam}

\begin{exam}\label{exam_action}
Let $G$ be a discrete group and $\cA=\oplus_{g\in G}\cA_g$ be a $G$-graded $*$-algebra. Every subspace $\cA_g\subseteq\cA$ has a natural structure of a right $\cB$-rigged left $\cB$-module where $\langle x,y\rangle_\cB:=y^*x,\ x,y\in\cA_g,$ see also \cite{fd}. Denote by $\cBp$ the set of proper q.m. $\cN\subseteq\cB$ inducible via $p$ such that $\Ind^p\cN$ is again proper.
For $\cN\in\cBp$ we say that ${}^g\!\cN$ is defined if $\Ind^{\cA_g}\cN$ is a proper q.m. in $\cB$. If ${}^g\!\cN$ is defined we put ${}^g\!\cN=\Ind^{\cA_g}\cN.$ It can be easily checked that ${}^g\!\cN\in\cBp.$ Let $\cN\in\cBp,$ then
$$\Ind\cN\cap\cB=\Res(\Ind\cN)=\bigcap {}^g \! \cN,$$
where the intersection is taken over all $g\in G$ such that ${}^g \! \cN$ is defined, cf. Proposition \ref{prop_graded}, (i).
\end{exam}

\begin{exam}
We illustrate the ``action'' defined in Example \ref{exam_action} on the Weyl algebra. We retain the notation from Section \ref{sect_weyl}.

Let $\lambda \in \dN_0$ and $k \in \dZ.$ We claim that ${}^k(\cN_\lambda)$ is defined and equal to $\cN_{\lambda-k}$ if and only if $\lambda\geq k.$ Namely, ${}^k(\cN_\lambda) = \{f(N) \in \cB_h \mid e_k^\ast e_k f(N-k) \in \cN_\lambda\}=
\{f(N) \in \cB_h \mid (e_k^\ast e_k)_{N=\lambda} f(\lambda-k) \ge 0\}$. The claim now follows from
the fact that
$(e_k^\ast e_k)_{N=\lambda}>0$ if $k\le \lambda$ and $(e_k^\ast e_k)_{N=\lambda}=0$ if $k>\lambda$.

Write $\cN_\infty$ for the set of all polynomials in $\cB_h = \RR[N]$ which have nonnegative leading coefficient
(i.e. they have strictly positive leading coefficient or they are identically zero).
Clearly $\cN_\infty$ is a quadratic module which contains all $e_k^\ast e_k$, $k \in \ZZ$,
hence $\cN_\infty$ is inducible. For every $k \in \ZZ$, we have that
$$
{}^k (\cN_\infty)=\cN_\infty.
$$
It follows that
$$
\Res(\Ind(\cN_\infty))=\bigcap_{k \in \ZZ}{}^k(\cN_\infty)=\cN_\infty.
$$

\end{exam}

\bibliographystyle{amsalpha}

\end{document}